\documentclass[letterpaper,twoside]{article}

\usepackage{amsmath,amssymb,amscd}
\usepackage{mathrsfs}
\usepackage{graphicx,xcolor}
\usepackage[latin1]{inputenc}

\newcommand{\lraup}{\relbar\joinrel\relbar\joinrel\rightharpoonup}

\newenvironment{proof}[1][]{\noindent{\bfseries Proof~#1\ }}{\hfill$\square$\medskip}

\let\bs\boldsymbol
\let\eps\varepsilon

\let\R\RR
\def\N{\mathbb N}
\def\C{\mathscr{C}}

\newcommand{\K}{\mathbb K}

\DeclareMathOperator{\sgn}{sgn}

\newtheorem{theorem}{Theorem}[section]
\newtheorem{lemma}[theorem]{Lemma}

\newtheorem{remark}[theorem]{Remark}
\begin{document}

\title{Solutions for linear conservation laws with gradient constraints}

\author{José Francisco Rodrigues \and Lisa Santos}

\date{}

\maketitle

\begin{abstract}
We consider variational inequality solutions with prescribed gradient constraints for first order linear boundary value problems. For operators with coefficients only in $L^2$, we show the existence and uniqueness of the solution by using a combination of parabolic regularization with a penalization in the nonlinear diffusion coefficient. We also prove  the continuous dependence of the solution with respect to the data, as well as,  in a coercive case, the asymptotic stabilization as time $t\rightarrow+\infty$ towards the stationary solution. In a particular situation, motivated by the transported sandpile problem, we give sufficient conditions for the equivalence of the first order problem with gradient constraint with a two obstacles problem, the obstacles being the signed distances to the boundary. This equivalence, in special conditions, illustrates also the possible stabilization of the solution in finite time.
\end{abstract}

\vspace{10mm}

\hfill{\large\em Dedicado a João Paulo Dias, no seu ativo septuagésimo aniversário! }

\vspace{5mm}

\section{Introduction}

Several works have developed solutions $u=u(x,t)$ to the linear equation of first order
\begin{equation} \label{1}
\partial_tu+\bs b\cdot\nabla u+cu=f,
\end{equation}
for $t>0$ and $x$ in an open subset $\Omega$ of $\R^N$, where $\bs b=\bs b(x,t)$ is a given vector field and $c=c(x,t)$ and $f=f(x,t)$ are given functions.

The well-known DiPerna and Lions theory of renormalized solutions, when $\bs b$ is given in Sobolev spaces, has been extended by Ambrosio to BV coefficients for the Cauchy problem and has found several applications in the study of hyperbolic systems of multidimensional conservation laws (see, for instance \cite{AcLOW2008}, for an introduction and references). The  initial-boundary value problem for \eqref{1} with a $C^1$ vector field $\bs b$ has been studied in the pioneer work of Bardos \cite{Bardos1970} using essentially a $L^2$ approach for the transport operator. This method also  holds for Lipschitz vector fields, as observed in \cite{Gaymonat-Leyland1987}, and was extended by Boyer \cite{Boyer2005} for solenoidal vector fields in Sobolev spaces that do not need to be tangential to the boundary of $\Omega$, i.e. $\bs b\cdot\bs n\neq 0$ on $\partial\Omega$ for $t>0$.

The delicate point is then to prescribe the boundary data to the normal trace of $\bs b$ on the portion of the space-time boundary $\Gamma_-\subset\partial\Omega\times(0,T)$ where the characteristics are entering the domain $Q_T=\Omega\times(0,T)$. In the case when $\Gamma_-$ does not vary with $t$, Besson and Pousin \cite{BessonPousin2007} have treated the initial-inflow problems for the continuity equation \eqref{1} with $\bs L^\infty$ velocity fields $\bs b$ with $c=\nabla\cdot\bs b=\text{div}\,\bs b$ also in $L^\infty(Q_T)$. Recently Crippa et al. \cite{CDS2014} have also considered this problem without that restriction on $\Gamma_-$ and with similar assumptions on $\bs b$ in BV.

Here we are interested in the initial-boundary value problem for \eqref{1} under the additional gradient constraint
\begin{equation}\label{2}
|\nabla u(x,t)|\le g(x,t),\quad (x,t)\in Q_T,
\end{equation}
where $g=g(x,t)$ is a given strictly positive and bounded function. This problem was already  considered in \cite{RodriguesSantos2012} in the framework of a quasilinear continuity equation
\begin{equation}\label{3}
\partial_tu+\nabla\cdot\bs\Phi(u)=F(u)
\end{equation}
and a Lipschitz semilinear lower order term $F=F(x,t,u)$, with a gradient bound in \eqref{2} that may depend also on the solution but not on time. As observed in \cite{RodriguesSantos2012}, in the linear transport equation \eqref{1}, corresponding to
$$\bs\Phi(u)=\bs b u\quad\text{and}\quad F(u)=f+\big(\nabla\cdot\bs b-c\big)u$$
with regular coefficients and $g=g(x)$ independent of $t$, the problem is well-posed in terms of a first order variational inequality with the convex set
\begin{equation}\label{4}
\K_g=\big\{v\in H^1_0(\Omega):|\nabla v(x)|\le g(x)\text{ a.e. }x\in\Omega\big\}.
\end{equation}
In \cite{RodriguesSantos2012} it is also proved the existence and asymptotic behaviour of quasivariational solutions for positive nonlinear gradient constraints $g=g(x,u)$ depending continuously on the solution $u=u(x,t)$. Here $H^1_0(\Omega)$ denotes the usual Sobolev space of functions vanishing on the boundary $\partial\Omega$, as the gradient bound allows to prescribe values on the whole boundary. Moreover, it allows also to consider the data $\bs b$, $c$ and $f$ only in $L^2(Q_T)$, provided $c-\frac12\nabla\cdot\bs b$ is bounded from below.

A motivation for the constraint \eqref{2} applied to the equation \eqref{1} is the ``transported sandpile'' problem. Following Prigozhin \cite{Prigozhin1994,Prigozhin1996}, the gradient of the shape of a growing pile of grains $z=z(x,t)$ characterized by its angle of repose $\alpha>0$ is constrained by its surface slope, i.e. $g=\arctan\alpha$. A general conservation of mass, in the form \eqref{3} with $\bs\Phi=-\mu\nabla u+\bs bu$ and source density $F$, with transport directed by $\bs b$ and dropping flow directed to the steepest descent $-\mu\nabla u$, should be then subjected to the unilateral conditions
$$\mu\ge0,\quad|\nabla u|\le g\quad\text{and}\quad|\nabla u|<g\Rightarrow\mu=0.$$

We illustrate this problem with the interesting example of the one dimensional special case announced in \cite{Rodrigues2012}: $\Omega=(0,1)$, $\bs b=1=g$, i.e. $\alpha=\frac\pi{4}$ and $f(x,t)=t$. Taking as initial condition the parabola $z_0(x)=-\frac12 x^2$, up to the point $\xi_0=\sqrt 3-1$, and the straight line $z_0(x)=x-1$, for $\xi_0\le x\le 1$, the profile of the ``transported sandpile''  growth attains a steady state exactly at $t=\frac54$. This happens with the first free boundary point $\xi(t)$ increasing from $\xi_0$ up to $t=\frac12$, touching then the boundary $x=1$, and decreasing till the midpoint $x=\frac12$. At this point, the free boundary $\xi(t)$ meets a second increasing free boundary $\zeta(t)=2(t-1)$, that appears at $t=1$ and increases up to the final stabilization at $t=\frac54.$

\begin{figure}[h]
$$\begin{array}{ccc}
\begin{minipage}{4cm} \includegraphics[height=3.75cm]{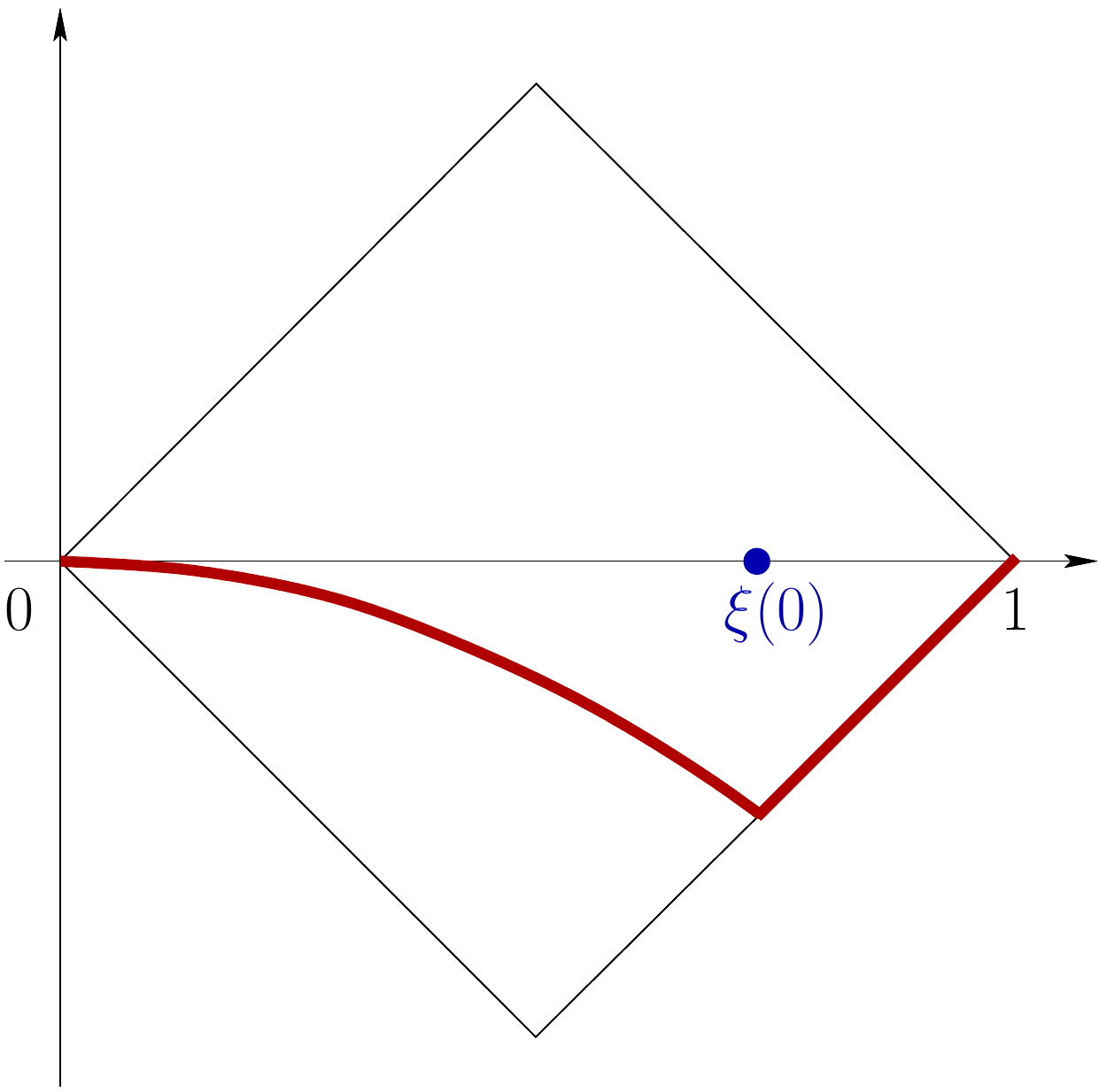}
\end{minipage}
\begin{minipage}{4cm} \includegraphics[height=3.75cm]{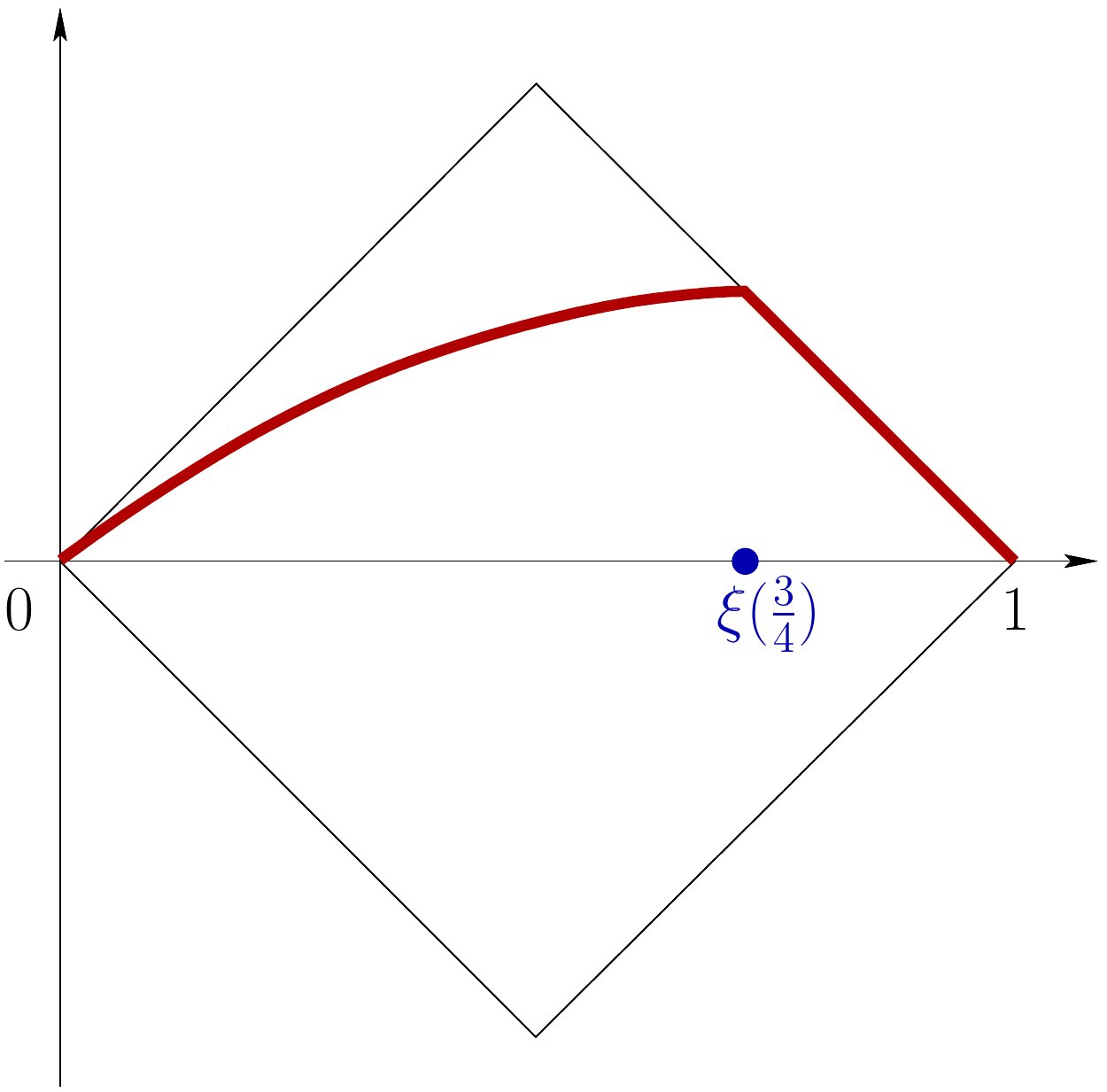}
\end{minipage}
\begin{minipage}{4cm} \includegraphics[height=3.75cm]{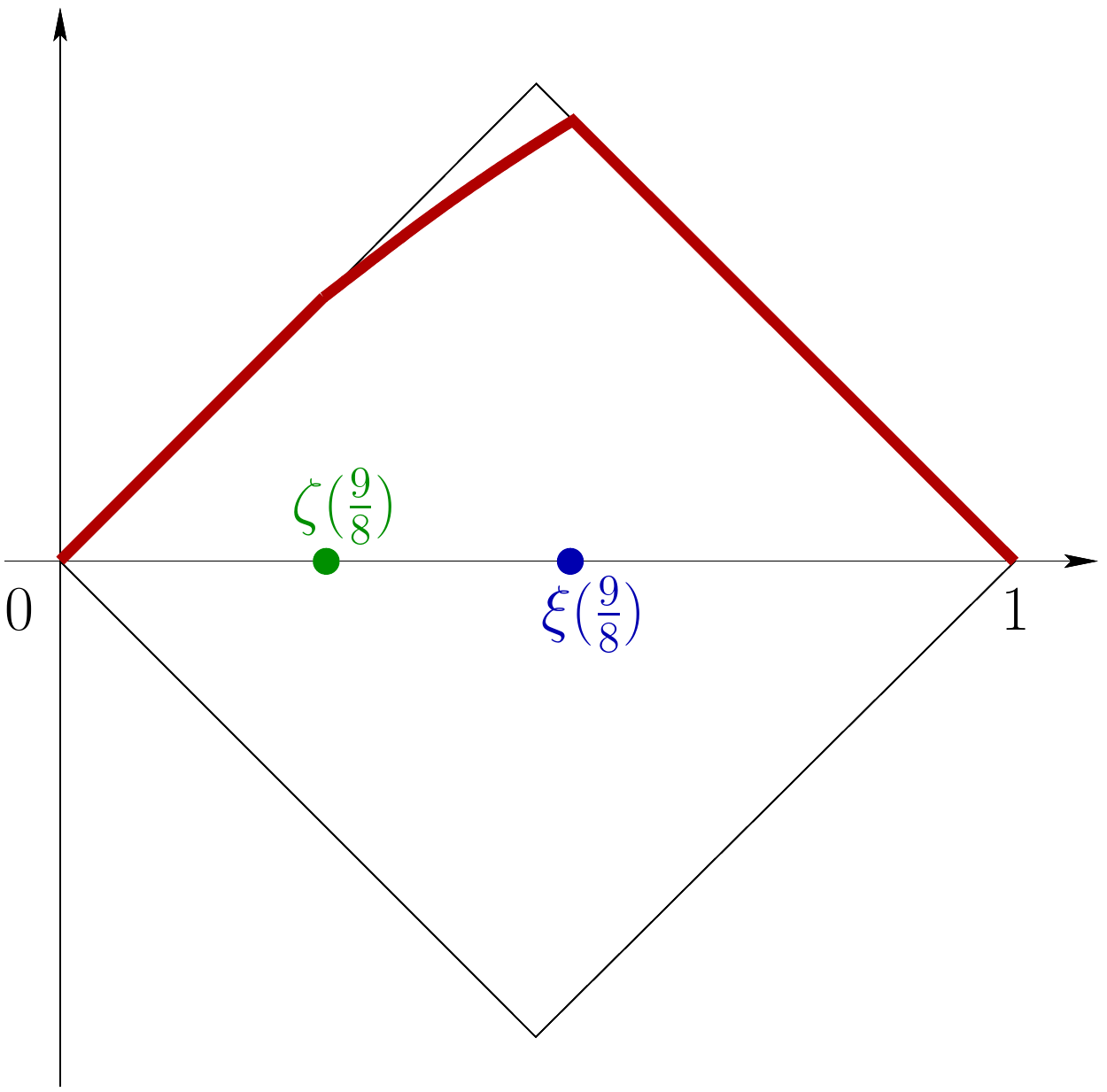}
\end{minipage}
\begin{minipage}{4cm} \includegraphics[height=3.75cm]{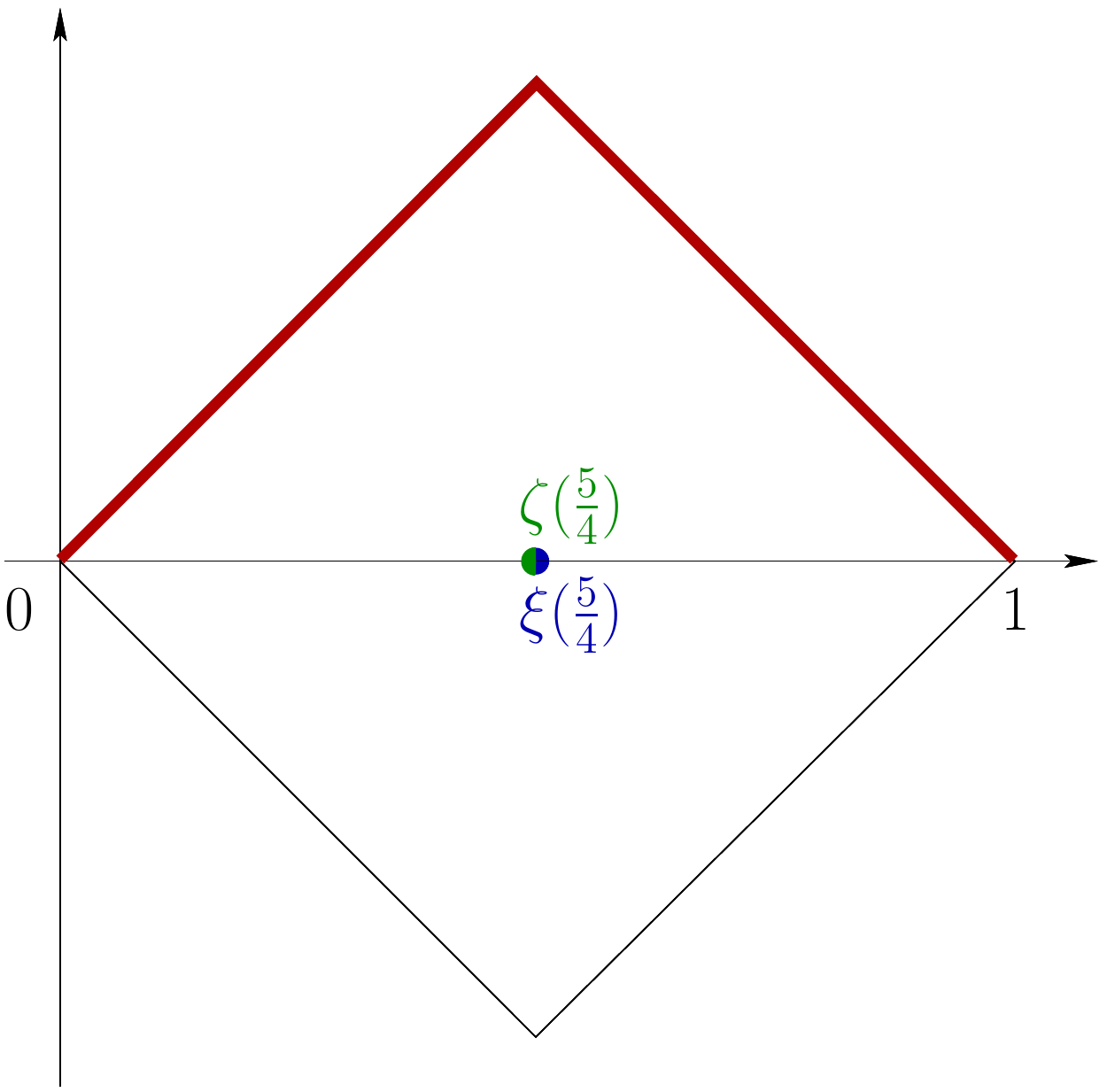}
\end{minipage}
\end{array}$$
\caption{The free boundary of the transported sandpile problem at $t=0, 3/4, 9/8$ and $5/4$.}
\end{figure}

The explicit sandpile profile is given by
$$z(x,t)=\left\{\begin{array}{ll}
tx- \frac12 x^2&\mbox{ if }0\le x\leq\xi(t)\mbox{ and }0\le t\le1,\vspace{3mm}\\
x-1&\mbox{ if }\xi(t)<x\le 1\mbox{ and }0\le t\le\frac12,\vspace{3mm}\\
1-x&\mbox{ if }\xi(t)<x\le 1\mbox{ and }\frac12< t\le1,\vspace{3mm}\\
x&\mbox{ if }0\le x\le \zeta(t)\mbox{ and }1< t\le\frac54,\vspace{3mm}\\
tx-\frac12 x^2&\mbox{ if }\zeta(t)< x\le \xi(t)\mbox{ and }1< t\le\frac54,\vspace{3mm}\\
x-1&\mbox{ if }\xi(t)< x\le 1\mbox{ and }1< t\le\frac54,\vspace{3mm}\\
\frac12-|x-\frac12|&\mbox{ if }t>\frac54,
\end{array}
\right.$$
where $\xi(t)=t-1+\sqrt{(1-t)^2+2}$,
if $0\le t\le\frac12$, and $\xi(t)=t+1-\sqrt{(t+1)^2-2}$, if $\frac12<t\le\frac54$.

It is clear that $z(t)\in\K_1\subset H^1_0(\Omega)$.

We introduce the function $d(x)=\frac12-|x-\frac12|$ and the convex set
$$\K_\vee^\wedge=\big\{v\in H^1_0(\Omega):-d(x)\le v(x)\le d(x) \text{ a.e. }x\in(0,1)\big\}\supset \K_1.$$

Since $\partial_tz+\partial_xz=t$ in $A=\big\{(x,t)\in Q_T:|\partial_xz(x,t)|<1\big\}$, by simple computation and integration in $Q_T$, we easily conclude that $z$, which (using $\vee=sup$ and $\wedge=inf$) can be written as
$$z(x,t)=\big(-d(x)\big)\vee\big((tx-\frac12 x^2)\wedge d(x)\big),$$
is then the unique solution in $\K_\vee^\wedge$ of the variational inequality
\begin{equation}\label{5}
\int_{Q_T}\big(\partial_tu+\partial_xu-t\big)(w-u)\ge 0\quad\forall\, w(t)\in\K_\vee^\wedge,\, 0<t<T,\qquad u(0)=z_0.
\end{equation}
But since $z_0,z(t)\in\K_1$, $z$ is also the solution of the variational inequality \eqref{5} with $w(t)\in\K_1\subset\K_\vee^\wedge$, which has at most one solution also in the convex set $\K_1$, defined as in \eqref{4} with $g\equiv1$.

In Section 2 we establish the existence and the uniqueness of the solution of the first order variational inequality associated with the general linear equation  \eqref{1} in a family of time dependent convex sets with gradient constraints of the type \eqref{4} with $g=g(x,t)$. We improve the results of  \cite{RodriguesSantos2012} under general square integrability assumptions on the coefficients and on the data, by direct estimates in the parabolic-penalized problem and passage to the limit, first in the penalization parameter $\eps$, and afterwards in the regularization parameter $\delta$. The continuous dependence of the solution with respect to the gradient constraint variations in $L^\infty$, to the coefficients of the operator and the data in  $L^1$, is proven in Section 3 under the weak coercive condition \eqref{coerc}, as well as the asymptotic convergence towards the unique stationary solution under the stronger coercive assumption \eqref{assumption_infty3}.

Finaly, in Section 4, we consider the special case of a constant vector $\bs b$, with $g=1$ and $f=f(t)$ bounded, to show the equivalence  of the variational inequalities with the gradient constraint and with the two obstacles, i.e. with the signed distances to the boundary constraints on the solution. This is a first result of this type for first order variational inequalities, similar to the elliptic well-known case of the elastoplastic torsion problem (see, for instance,  \cite{rod87} and its references) and to the parabolic case without convection considered in  \cite{Lisa1991, Lisa2002}, where it was shown that this equivalence is not always possible in the general case. With additional conditions, that include the above one dimensional transported sand pile problem, we establish the finite time stabilization of the solution. This extends to the convective problem a similar result by Cannarsa et al. \cite{Cannarsa} and raises the interesting open question of establishing more general conditions on the finite time stabilization of evolutionary problems with gradient constraints.

\section{Existence and uniqueness of the variational solution}

Let $\Omega $ be a bounded open subset of $\R^N$ with a Lipschitz boundary $\partial\Omega $ and, for any  $T>0$, denote
 $Q_T=\Omega \times(0,T)$.

Assume that
\begin{equation}\label{bc}
\bs b\in\bs L^2(Q_T)\quad\text{and}\quad c\in L^2(Q_T),
\end{equation}
and there exists $l\in\R$ such that
\begin{equation}\label{coerc}
c-\frac12\nabla\cdot\bs b\ge l\quad\text{in}\quad Q_T,
\end{equation}
being this inequality satisfied in the distributional sense, since $\nabla\cdot\bs b$ does not need to be a function.

In addition we also suppose given

\begin{equation}\label{dados}
f\in L^2(Q_T)\quad\text{and}\quad u_0\in\K_{g(0)},
\end{equation}
with
\begin{equation}\label{g}
g\in W^{1,\infty}\big(0,T;L^\infty(\Omega)\big),\quad g\ge m>0.
\end{equation}

As in  \eqref{4}, we define, for $t\ge0,$
\begin{equation*}
\K_{g(t)}=\big\{v\in H^1_0(\Omega ):|\nabla  v(x)|\leq g(x,t)\mbox{ for a.e. in }x\in\Omega \big\}.
\end{equation*}

Consider the following variational inequality problem: To find $u$, in an appropriate space, such that
\begin{equation}
\label{iv}
\begin{array}{l}
\left\{\begin{array}{l}
u(t)\in\K_{g(t)}\mbox{ for a.e. }t\in (0,T),\qquad  u(0)=u_0,\\
\\
\displaystyle\int_\Omega  \partial_t u(t)(v-u(t))+\displaystyle\int_\Omega\bs b(t)\cdot\nabla u(t)(v-u(t))+\displaystyle\int_\Omega c(t)\,u(t)(v-u(t))
\geq\displaystyle\int_\Omega  f(t)(v-u(t)),\\
\\
\hfill{\forall \, v\in\K_{g(t)},\mbox{ for a.e. }t\in (0,T)}.
\end{array}
\right.
\end{array}
\end{equation}

\begin{theorem}
	\label{main}
	With the assumptions \eqref{bc}-\eqref{g}, problem \eqref{iv} has a unique solution
	\begin{equation*}
	u\in L^\infty\big(0,T;W^{1,\infty}_0(\Omega)\big)\cap \C(\overline Q_T),\qquad \partial_tu\in L^2(Q_T).
	\end{equation*}
\end{theorem}

\begin{proof}
To prove the uniqueness of the solution we assume there exist two solutions $u_1$ and $u_2$. Using $u_2=u_2(t)$ as test function in \eqref{iv} for the variational inequality of $u_1$ and reciprocally, setting $\bar u=u_1-u_2$ at a.e. $t>0$, we obtain
$$\int_\Omega\partial_t\bar u(t)\bar u(t)+\int_\Omega\bs b(t)\cdot\nabla\bar u(t)\,\bar u(t)+\int_\Omega c(t)\bar u^2(t)\le0. $$
Using \eqref{coerc}, for any $v\in\C^\infty_0(\Omega)$, we have
$$\frac12\int_\Omega\bs b(t)\cdot\nabla v^2 + \int_\Omega c(t)v^2\ge l\int_\Omega v^2$$
and, by approximation in $H^1_0(\Omega)$ of $\bar u(t)$, we obtain,
$$\frac{d\ }{dt}\int_\Omega |\bar u(t)|^2+2l\int_\Omega|\bar u(t)|^2\le0.$$
By Gronwall's inequality, we conclude $\bar u\equiv0$ from $\bar u(0)=0$.

To prove the existence of a solution, we consider a family of approximating quasilinear parabolic problems for $u^{\eps\delta}$, with $\eps,\delta\in (0,1)$, defined as follows
\begin{equation}
\label{apr}
\left\{\begin{array}{l}
\partial_tu^{\eps\delta}-\delta\nabla \cdot( k_\eps(|\nabla u^{\eps\delta}|^2-g^2)\nabla u^{\eps\delta})+\bs b^\delta\cdot\nabla u^{\eps\delta}+c^\delta\, u^{\eps\delta}=f^\delta \mbox{ in }Q_T,\vspace{3mm}\\
u^{\eps\delta}=0\mbox{ on }\partial\Omega\times(0,T),\vspace{3mm}\\
u^{\eps\delta}(0)={u_0^\eps}\mbox{ in }\Omega _0,
\end{array}
\right.
\end{equation}
where $\bs b^\delta$, $c^\delta$, $f^\delta$ and $u_0^\eps$ are $\C^\infty$ appropriate regularizations of $\bs b$, $c$, $f$ and $u_0$, respectively,
with $|\nabla u_0^\eps|\leq g(0)$ and $k_\eps$ is a smooth real function such that $k_\eps(s)=1$ if $s\le 0$ and $k_\eps(s)=e^\frac{s}\eps$ if $s\ge\eps$. Notice that this problem has a unique solution $u^{\eps\delta}\in H^1\big(0,T;L^2 (\Omega)\big) \cap L^\infty\big(0,T;H^{1}_0(\Omega)\big)\cap \C(\overline{Q}_T)$,
 by the classical theory of parabolic quasilinear problems (see, for instance,  \cite{LSU67}).

We prove first several {\em a priori} estimates.

\vspace{3mm}

\noindent{\bf Estimate 1}
\begin{equation}\label{k}
\|k_\eps(|\nabla u^{\eps\delta}|^2-g^2)\|_{L^1
	(Q_T)}\leq \frac1\delta C_1,
\end{equation}
for some   constant $C_1$ dependent only on $m$, $|l|$, $\|f\|_{L^2(Q_T)}$, $\|g\|_{L^2(Q_T)}$ and $\|u_0\|_{L^2(\Omega)}$.

\vspace{3mm}

Multiplying the equation of the problem \eqref{apr} by $u^{\eps\delta}$ and integrating over
$Q_t=\Omega \times]0,t[$, we have
\begin{equation*}
\frac 12\int_\Omega |u^{\eps\delta}(t)|^2+\delta\int_{Q_t}k_\eps(|\nabla u^{\eps\delta}|^2-g^2)|\nabla u^{\eps\delta}|^2+\int_{Q_t}\big(\bs b^\delta\cdot\nabla u^{\eps\delta}\big)\,u^{\eps\delta}+
\int_{Q_t}c^\delta\,|u^{\eps\delta}|^2=\int_{Q_t}f^\delta u^{\eps\delta}+
\frac 12\int_\Omega| {u_0^\eps}|^2.
\end{equation*}

Observing that
\begin{equation*}
\int_{Q_t} \big(\bs b^\delta\cdot\nabla u^{\eps\delta}\big)\,u^{\eps\delta}=-\frac12\int_{Q_t}\big(\nabla\cdot\bs b^\delta\big)\,\big(u^{\eps\delta}\big)^2,
\end{equation*}
and using the coercive inequality for the regularized coefficients
$$c^\delta-\tfrac12\nabla\cdot\bs b^\delta=(c-\tfrac12\nabla\cdot\bs b)*\rho_\delta\ge l*\rho_\delta=l,$$
we have
\begin{equation*}
\frac 12\int_\Omega |u^{\eps\delta}(t)|^2+\delta\int_{Q_t}k_\eps(|\nabla u^{\eps\delta}|^2-g^2)|\nabla u^{\eps\delta}|^2\le \|f^\delta\|_{L^2(Q_T)}\|u^{\eps\delta}\|_{L^2(Q_t)}+\frac12\|u_0\|_{L^2(\Omega)}^2+|l|\|u^{\eps\delta}\|_{L^2(Q_t)}^2.
\end{equation*}

Hence, by the integral Gronwall's inequality, there exists a positive constant $C_T$, independent of $\eps$ and $\delta$, such that
$$\|u^{\eps\delta}\|_{L^2(Q_T)}\le C_T$$
and so
\begin{equation*}
\delta\int_{Q_t}k_\eps(|\nabla u^{\eps\delta}|^2-g^2)|\nabla u^{\eps\delta}|^2\le  C',
\end{equation*}
where $C'=C'(\|f\|_{L^2(Q_T)},\|u_0\|_{L^2(\Omega))}, |l|)$.

On the other hand, we observe that
\begin{equation}\label{11}
\int_{Q_T}k_\eps(|\nabla u^{\eps\delta}|^2-g^2)|\nabla u^{\eps\delta}|^2=\int_{Q_T}k_\eps(|\nabla u^{\eps\delta}|^2-g^2)\big(|\nabla u^{\eps\delta}|^2-g^2\big)+
\int_{Q_T}k_\eps(|\nabla u^{\eps\delta}|^2-g^2)g^2.
\end{equation}

Since $k_\eps(s)=1$ for $s\le 0$ and $k_\eps(s)s\ge 0$, for all $s\ge0$, then
\begin{multline}\label{22}
\int_{Q_t}k_\eps(|\nabla u^{\eps\delta}|^2-g^2)\big(|\nabla u^{\eps\delta}|^2-g^2\big)=
\int_{\{|\nabla u^{\eps\delta}|^2\le g^2\}} k_\eps(|\nabla u^{\eps\delta}|^2-g^2)\big(|\nabla u^{\eps\delta}|^2-g^2\big)\\
+\int_{\{|\nabla u^{\eps\delta}|^2>g^2\}}k_\eps(|\nabla u^{\eps\delta}|^2-g^2)\big(|\nabla u^{\eps\delta}|^2-g^2\big)\ge
-\int_{Q_T}g^2.
\end{multline}

From \eqref{11} and \eqref{22} we obtain
\begin{equation*}
\int_{Q_t}k_\eps(|\nabla w|^2-g^2)
\leq \frac1{m^2}
\left(\int_{Q_T}k_\eps(|\nabla w|^2-g^2)|\nabla w|^2+\int_{Q_T}g^2\right)
\le\frac1{ m^2}\Big(\frac1\delta C'+\|g\|^2_{L^2(Q_T)}\Big)\le \frac1\delta C_1,
\end{equation*}
where $C_1=C_1(m,\|f\|_{L^2(Q_T)},\|g\|_{L^2(Q_T)},\|u_0\|_{L^2(\Omega)},|l|)$.

\vspace{3mm}

\noindent{\bf Estimate 2}
\begin{equation}
\label{lp}
\|\nabla u^{\eps\delta}\|_{\bs L^p(Q_T)}\leq D_\delta,
\end{equation}
where, for any $\delta>0$ and any $1\le p<\infty$, the  constant $D_\delta$ depends only on $p$,  $m$, $|l|$, $\|f\|_{L^2(Q_T)}$, $\|g\|_{L^2(Q_T)}$, $\|u_0\|_{L^2(\Omega)}$ and a negative power of $\delta$.

\vspace{3mm}

From \eqref{k} we know that
\begin{equation*}
\int_{Q_T}k_\eps(|\nabla u^{\eps\delta}|^2-g^2)\leq\frac1\delta C_1,
\end{equation*}
where $C_1$ is the positive constant of Estimate 1. So,
\begin{equation*}
\frac1\delta C_1\geq\int_{\{|\nabla u^{\eps\delta}|^2>g^2+\eps\}}k_\eps(|\nabla u^{\eps\delta}|^2-
g^2)=
\int_{\{|\nabla u^{\eps\delta}|^2>g^2+\eps\}}e^{\frac{|\nabla u^{\eps\delta}|^2-g^2}{\eps}}
\end{equation*}
and recalling that, for all $ s>0$ and all $j\in\N$, $e^s\geq\frac{s^j}{j!}$,
we get, for any $j\in\N$,
\begin{equation*}
\int_{\{|\nabla u^{\eps\delta}|^2>g^2+\eps\}}
\left(|\nabla u^{\eps\delta}|^2-g^2\right)^j\leq
j!\eps^j\int_{\{|\nabla u^{\eps\delta}|^2>g^2+\eps\}} e^{\frac{|\nabla u^{\eps\delta}|^2-g^2}{\eps}}\leq \frac1\delta j!\eps^jC_1.
\end{equation*}

Given $1\le p<\infty$, we have
\begin{equation}
\label{aaa}
\int_{Q_T}|\nabla u^{\eps\delta}|^{2p}=\int_{\{|\nabla u^{\eps\delta}|^2\leq
	g^2+\eps\}}|\nabla u^{\eps\delta}|^{2p}+
\int_{\{|\nabla u^{\eps\delta}|^2> g^2+\eps\}}|\nabla u^{\eps\delta}|^{2p}
\end{equation}
and, since $g$ is bounded, we  can
estimate, for any $p\in\N$, the second integral in the second term of \eqref{aaa} as follows,
\begin{multline*}
\displaystyle\int_{\{|\nabla u^{\eps\delta}|^2> g^2+\eps\}}|\nabla u^{\eps\delta}|^{2p}\leq \displaystyle{\int_{\{|\nabla u^{\eps\delta}|^2>
		g^2+\eps\}}\sum_{j=0}^p}
\binom{p}{j}
\displaystyle{\|g\|_{L^\infty(Q_T)}^{2p-2j}\left(|\nabla u^{\eps\delta}|^2-g^2\right)^j}\\
\displaystyle{\leq\frac1\delta\sum_{j=0}^p}
\binom{p}{j}
\|g\|_{L^\infty(Q_T)}^{2p-2j}j!\eps^jC_1.
\end{multline*}

The first integral in the second term of \eqref{aaa} is clearly bounded since
\begin{equation*}
\int_{\{|\nabla u^{\eps\delta}|^2\leq g^2+\eps\}}|\nabla u^{\eps\delta}|^{2p}\leq\int_{Q_T}\left(g^2+1\right)^{p}
\end{equation*}
and the conclusion follows easily, first for $2p\in\N$ and afterwards  for any $1\le p<\infty$.

\vspace{3mm}

\noindent{\bf Estimate 3}
\begin{equation}
\label{te}
\|\partial_t u^{\eps\delta}\|_{L^2(Q_T)}^2\leq 4\big(\|\bs b^\delta\|^2_{\bs L^s(Q_T)}+C_3\|c^\delta\|^2_{\bs L^s(Q_T)}\big)\|\nabla u^{\eps\delta}\|^2_{\bs L^q(Q_T)}+ C_4,
\end{equation}
where, for $2<s<\frac{2N}{N-2}$, $C_3$ is an upper bound, independent of $q=\frac{2s}{s-2}$, of the Poincaré constant for $W^{1,q}_0(\Omega)$ and $C_4$ is a positive constant depending only on $m$,  $|l|$, $\|f\|_{L^2(Q_T)}$, $\|g\|^2_{W^{1,\infty}(0,T;L^\infty(\Omega ))}$ and $\|u_0\|_{L^2(\Omega)}$.

\vspace{3mm}

We multiply the equation of problem \eqref{apr} by $\partial_tu^{\eps\delta}$  and we integrate over $Q_t$, noting that $\partial_t u^{\eps\delta}=0$ on $\partial\Omega\times(0,T)$. Denoting $\phi_\eps(s)=\displaystyle{\int_0^sk_\eps(\tau)d\tau}$, we have
\begin{multline*}
\int_{Q_t}|\partial_tu^{\eps\delta}|^2+\frac \delta2\int_{Q_t}\frac{d\ }{dt}\left(\phi_\eps(|\nabla u^{\eps\delta}|^2-g^2)\right)
+\delta\int_{Q_t}k_\eps(|\nabla u^{\eps\delta}|^2-g^2)g \partial_t g\\
+\int_{Q_t}(\bs b^\delta\cdot\nabla u^{\eps\delta})\,\partial_tu^{\eps\delta}+
\int_{Q_t}c^\delta\,u^{\eps\delta}\,\partial_tu^{\eps\delta}=\int_{Q_t}f^\delta \partial_t u^{\eps\delta}.
\end{multline*}

We choose $2<s<\frac{2N}{N-2}$ and  $q=\frac{2s}{s-2}$, and so we have $\frac1{s}+\frac1{q}+\frac12=1$. Then
\begin{equation*}
\Big|\int_{Q_t}(\bs b^\delta\cdot\nabla u^{\eps\delta})\,\partial_tu^{\eps\delta}\Big|\le\|\bs b^\delta\|_{L^s(Q_T)}\|\nabla u^{\eps\delta}\|_{L^q(Q_T)}\|\partial_tu^{\eps\delta}\|_{L^2(Q_T)}\le \|\bs b^\delta\|_{L^s(Q_T)}^2\|\nabla u^{\eps\delta}\|_{L^q(Q_T)}^2+\frac14\|\partial_tu^{\eps\delta}\|^2_{L^2(Q_T)},
\end{equation*}
and
\begin{equation*}
\Big|\int_{Q_t}c^\delta\,u^{\eps\delta}\,\partial_tu^{\eps\delta}\Big|\le\|c^\delta\|_{L^s(Q_T)}\| u^{\eps\delta}\|_{L^q(Q_T)}\|\partial_tu^{\eps\delta}\|_{L^2(Q_T)}
\le \|c^\delta\|_{L^s(Q_T)}^2\| u^{\eps\delta}\|_{L^q(Q_T)}^2+\frac14\|\partial_tu^{\eps\delta}\|^2_{L^2(Q_T)}.
\end{equation*}
So
\begin{align*}
\frac14\int_{Q_T}|\partial_tu^{\eps\delta}|^2\leq&
\|f^\delta\|_{L^2(Q_T)}^2+\frac\delta2\int_\Omega\phi_\eps\big(|\nabla u^{\eps\delta}(0)|^2-g^2(0)\big)-\frac\delta2\int_\Omega
\phi_\eps\big(|\nabla u^{\eps\delta}(t)|^2-g^2(t)\big)\\
&+C_1\|g\|_{L^\infty(Q_T)}\|\partial_tg\|_{L^\infty(Q_T)}+\big(\|\bs b^\delta\|_{L^s(Q_T)}^2+C_q\|c^\delta\|_{L^s(Q_T)}^2\big)\|\nabla u^{\eps\delta}\|_{L^q(Q_T)}^2,
\end{align*}
being $C_q$ a Poincaré constant. Observe that, since $\Omega$ is bounded we may find a positive upper bound $C_3$ of $C_q$, independently of $q\le\infty$.

On one hand
$$\displaystyle\int_\Omega\phi_\eps\big(|\nabla u^{\eps\delta}(0)|^2-g^2(0)\big)\le 0,$$
because $|\nabla u^{\eps\delta}(0)|=|\nabla {u_0^\eps}|\leq g(0)$.
On the other hand, if we set $\Lambda=\{(x,t)\in Q_T:|\nabla u^{\eps\delta}(x,t)|<g(x,t)\}$, we have
\begin{align*}
&\qquad \phi_\eps (|\nabla u^{\eps\delta}(x,t)|^2-g^2(x,t))=|\nabla u^{\eps\delta}(x,t)|^2-g^2(x,t)\ge -g^2(x,t)\qquad \mbox{for a.e. }(x,t)\in\Lambda,\\
&\qquad \phi_\eps (|\nabla u^{\eps\delta}(x,t)|^2-g^2(x,t))\ge 0\ge -g^2(x,t)\qquad\mbox{for a.e. }(x,t)\in Q_T\setminus\Lambda.
\end{align*}
Consequently, for a.e. $t\in (0,T)$,
\begin{equation*}
-\int_\Omega \phi_\eps (|\nabla u^{\eps\delta}(t)|^2-g^2(t))\le\|g\|^2_{L^\infty(0,T;L^2(\Omega ))}.
\end{equation*}

So,
\begin{multline*}
\frac14\|\partial_tu^{\eps\delta}\big\|^2_{L^2(Q_T)}\le\|f^\delta\|^2_{L^2(Q_T)}+\frac\delta{2}\|g\|^2_{L^\infty(0,T;L^2(\Omega))}\\
+C_1\|g\|_{L^\infty(Q_T)}\|\partial_t g\|_{L^\infty(Q_T)}+\big(\|\bs b^\delta\|_{L^s(Q_T)}^2+C_q\|c^\delta\|_{L^s(Q_T)}^2\big)\|\nabla u^{\eps\delta}\|_{L^q(Q_T)}^2,
\end{multline*}
and the proof of Estimate 3 is concluded.

\vspace{3mm}

By \eqref{lp} and \eqref{te}, we know there exist
constants $D_\delta$, $C_\delta$ and $C_4$, independent of $\eps$, such that, for each $N<p<\infty$,
\begin{equation*}
\|u^{\eps\delta}\|_{L^p(0,T;W^{1,p}_0(\Omega ))}\le D_\delta,\qquad
\|\partial_tu^{\eps\delta}\|_{L^2(Q_T)}\le \big(\|\bs b^\delta\|_{\bs L^s(Q_T)}+\|c^\delta\|_{ L^s(Q_T)}\big)C_\delta+C_4.
\end{equation*}

Since $u^{\eps\delta}$ is bounded in $H^1\big(0,T;L^2(\Omega)\big)\cap L^p\big(0,T;W^{1,p}_0(\Omega)\cap\C^{0,1-N/p}(\overline{\Omega })\big)$, independently of $\eps\in(0,1)$, for $p>N$, by a known compactness theorem (\cite{Simon1987}, page 84), $\{u^{\eps\delta}\}_\eps$ is relatively compact
in $\C\big([0,T];\C(\overline{\Omega })\big)$. Then,
at least for a subsequence,
\begin{equation*}
u^{\eps\delta}\underset{\eps\rightarrow0}{\longrightarrow} u^\delta\qquad \mbox{  in }\C(\overline Q_T).
\end{equation*}

The above estimates also imply that we may choose, always with fixed $\delta$,
\begin{equation*}
u^{\eps\delta}\underset{\eps\rightarrow0}{\lraup} u^\delta\quad \mbox{ weakly in }L^p(0,T;W^{1,p}_0\Omega )),\ 1\le p<\infty,\qquad
\partial_tu^{\eps\delta}\underset{\eps\rightarrow0}{\lraup}\partial_tu^\delta\quad \mbox{ weakly in }L^2(Q_T).
\end{equation*}

Given $v\in
L^\infty(0,T;H^1_0(\Omega ))$ such that $v(t)\in\K_{g(t)}$ for a.e. $t\in(0,T)$, we multiply the equation of problem \eqref{apr} by $v(t)-u^{\eps\delta}(t)$, we use the
monotonicity of $k_\eps$ and we integrate over $\Omega\times(s,t) $, $0\leq s<t\leq T$, to conclude that
\begin{multline*}
\int_s^t\int_\Omega \partial_tu^{\eps\delta}(v-u^{\eps\delta})+\delta\int_s^t\int_\Omega \nabla v\cdot\nabla
(v-u^{\eps\delta})\\
+\int_s^t\int_\Omega\bs b^\delta\cdot\nabla u^{\eps\delta}(v-u^{\eps\delta})+\int_s^t\int_\Omega c u^{\eps\delta}(v-u^{\eps\delta})
\ge\int_s^t\int_\Omega f^\delta (v-u^{\eps\delta}).
\end{multline*}

Letting $\eps\rightarrow 0$, since $s$ and $t$ are arbitrary, we obtain that for a.e. $t\in(0,T)$
\begin{multline*}
\int_\Omega  \partial_tu^\delta(t)(v(t)-u^\delta(t))+\delta\int_\Omega \nabla  v(t)\cdot\nabla (v(t)-u^\delta(t))\\
+\int_\Omega\bs b^\delta(t)\cdot\nabla u^\delta(t)\big(v(t)-u^\delta(t)\big)+\int_\Omega c^\delta(t)\, u^\delta(t)\big(v(t)-u^\delta(t)\big)
\ge\int_\Omega  f^\delta(t)\big(v(t)-u^\delta(t)\big),
\end{multline*}
for all $v\in L^\infty(0,T;H^1_0(\Omega ))$, such that, $v(t)\in\K_{g(t)}$ for a.e. $t\in(0,T)$.

Set $A_{\eps}=\{(x,t)\in Q_T:|\nabla u^{\eps\delta}(x,t)|^2-g^2(x,t)\ge \sqrt\eps\}$.
 Since  $k_\eps(|\nabla u^{\eps\delta}|^2-g^2)\ge \displaystyle e^{\frac{1}{\sqrt \eps}}$
in $A_{\eps}$, then we have
\begin{equation*}
\left|A_{\eps}\right|=\int_{A_{\eps}}1\le\int_{A_{\eps}}\displaystyle{\frac{k_\eps(|\nabla u^{\eps\delta}|^2-g^2)}{e^{\frac{1}{\sqrt \eps}}}}
\le \frac{C_1}\delta e^{-\frac{1}{\sqrt \eps}},
\end{equation*}
by  \eqref{k}, being $C_1$ a constant independent of $\eps$ as we have seen. So we have
\begin{multline}\label{nabla_u}
\int_{Q_T}\left(|\nabla  u^\delta|^2-g^2\right)^+\le\liminf_{\eps\rightarrow 0}\int_{Q_T}\left(|\nabla u^{\eps\delta}|^2-g^2-\sqrt{\eps}\right)^+\\
=\liminf_{\eps\rightarrow 0}\int_{A_{\eps}}\left(|\nabla u^{\eps\delta}|^2-g^2-
\sqrt{\eps}\right)\le\lim_{\eps\rightarrow 0}\,M_\delta\,\left|A_{\eps}\right|^\frac 12=0,
\end{multline}
where $M_\delta$ is an upper bound of $\|\,|\nabla u^{\eps\delta}|^2-g^2-\sqrt{\eps}\,\|_{L^2(Q_T)}$,
independent of $\eps$. Consequently,
\begin{equation*}
|\nabla  u^\delta|\le g\qquad \mbox{ a.e. in }Q_T
\end{equation*}
and so $u^\delta(t)\in\K_{g(t)}$ for a.e. $t\in(0,T)$. Let $z\in L^\infty(0,T;H^1_0(\Omega ))$ be such that $z(t)\in\K_{g(t)}$.
Defining $v=u+\theta(z-u)$,
$\theta\in(0,1]$, then $v(t)\in\K_{g(t)}$. Using $v(t)$ as test function in \eqref{iv} and dividing both sides of the inequality
by $\theta$, we get
\begin{multline*}
\int_\Omega  \partial_tu^\delta(t)(z(t)-u^\delta(t))+\delta\int_\Omega \nabla
u^\delta(t)\cdot\nabla (z(t)-u^\delta(t))+\delta\,\theta\int_\Omega  |\nabla (z(t)-u^\delta(t))|^2\\
\int_\Omega\bs b^\delta(t)\cdot\nabla u^{\delta}(t)(z(t)-u^{\delta}(t))+\int_\Omega c^\delta(t)\,u^{\delta}(t) (z(t)-u^{\delta}(t))
\ge\int_\Omega  f^\delta(t)(z(t)-u^\delta(t))
\end{multline*}
and, letting $\theta\rightarrow  0$, we conclude that $u^\delta$ solves the following variational inequality
\begin{equation}
\label{ivdelta}
\begin{array}{l}
\left\{\begin{array}{l}
u^\delta(t)\in\K_{g(t)}\mbox{ for a.e. }t\in (0,T),\qquad  u^\delta(0)=u_0,\\
\\
\displaystyle\int_\Omega  \partial_t u^\delta(t)(v-u^\delta(t))+\delta\int_\Omega\nabla u^\delta(t)\cdot\nabla(v-u^\delta(t))\\
\\
\hspace{1cm}+
\displaystyle\int_\Omega\bs b^\delta(t)\cdot\nabla u^\delta(t)(v-u^\delta(t))+\displaystyle\int_\Omega c^\delta(t)\,u^{\delta}(t) (v-u^\delta(t))
\geq\displaystyle\int_\Omega  f^\delta(t)(v-u^\delta(t)),\\
\\
\hspace{8cm}{\forall \, v\in\K_{g(t)},\mbox{ for a.e. }t\in (0,T)}.
\end{array}
\right.
\end{array}
\end{equation}
Recalling the Estimate 3  we have
\begin{align*}
\|\partial_t u^{\eps\delta}\|_{L^2(Q_T)}^2&\le \big(\|\bs b^\delta\|^2_{\bs L^s(Q_T)}+C_3\|c^\delta\|^2_{ L^s(Q_T)}\big)\|\nabla u^{\eps\delta}\|^2_{\bs L^{\frac{2s}{s-2}}(Q_T)}+C_4\\
&\leq\big(\|\bs b^\delta\|^2_{\bs L^s(Q_T)}+C_3\|c^\delta\|^2_{ L^s(Q_T)}\big)\Big(\int_{Q_T}(|\nabla u^{\eps\delta}|^\frac{2s}{s-2}-g^\frac{2s}{s-2}-\sqrt\eps)^++\int_{Q_T}(g^\frac{2s}{s-2}+\sqrt\eps)\Big)^\frac{s-2}{s}+C_4.
\end{align*}
Passing to the $\liminf$ when $\eps\rightarrow0$ and arguing as in \eqref{nabla_u}, we conclude that
$$\liminf_{\eps\rightarrow0}\int_{Q_T}(|\nabla u^{\eps\delta}|^\frac{2s}{s-2}-g^\frac{2s}{s-2}-\sqrt\eps)^+=0$$
and, consequently,
$$\|\partial_t u^{\delta}\|_{L^2(Q_T)}\le
4\big(\|\bs b^\delta\|^2_{\bs L^s(Q_T)}+C_3\|c^\delta\|^2_{ L^s(Q_T)}\big)\|g\|_{L^{\frac{2s}{s-2}}(Q_T)}^2+C_4.$$
Observing that $$\big(\|\bs b^\delta\|^2_{\bs L^s(Q_T)}+C_3\|c^\delta\|^2_{ L^s(Q_T)}\big)\|g\|_{L^{\frac{2s}{s-2}}(Q_T)}^2\underset{s\rightarrow2}\longrightarrow \big(\|\bs b^\delta\|^2_{\bs L^2(Q_T)}+C_3\|c^\delta\|_{L^2(Q_T)}\big)\|g\|_{L^\infty(Q_T)}^2,$$ we have the sequence $\{\partial_t u^{\delta}\}_\delta$ uniformly bounded in $L^2(Q_T)$.

Moreover, the sequence $\{u^\delta\}_\delta$ is uniformly bounded in $L^\infty\big(0,T;W^{1,\infty}_0(\Omega)\big)$, independently of $\delta$, since each
$u^\delta(t)$ belongs to $\K_{g(t)}$. So, there exists a function $u\in  L^\infty\big(0,T;W^{1,\infty}_0(\Omega)\big)\cap H^1\big(0,T;L^2(\Omega)\big)\cap \C(\overline Q_T)$ and,
at least for a subsequence,
\begin{equation*}
u^{\delta}\underset{\delta\rightarrow 0}{\longrightarrow} u\qquad \mbox{  in }\C(\overline Q_T),
\end{equation*}
\begin{equation*}
u^{\delta}\underset{\delta\rightarrow 0}{\lraup} u\quad \mbox{ weakly in }L^p\big(0,T;W^{1,p}_0(\Omega )\big),\ 1\le p<\infty,\qquad
\partial_tu^{\delta}\underset{\delta\rightarrow 0}{\lraup}\partial_tu\quad \mbox{ weakly in }L^2(Q_T).
\end{equation*}
Integrating in \eqref{ivdelta} between $s$ and $t$, for $0\le s<t\le T$, and passing to the limit when $\delta\rightarrow0$, we get
\begin{equation*}
\int_s^t\int_\Omega  \partial_t u(v-u)+\int_s^t\int_\Omega\bs b\cdot\nabla u(v-u)+\int_s^t\int_\Omega c\,u (v-u)
\geq\int_s^t\int_\Omega  f(v-u),
\end{equation*}
for all $v$ such that $v(t)\in\K_{g(t)}$ for a.e. $t\in(0,T)$. Since $s$ and $t$ are arbitrary, we can drop the integration in time.
Since  $u^\delta(t)\in\K_{g(t)}$ for a.e. $t\in(0,T)$, the same holds for $u(t)$, concluding that $u$ solves the variational inequality \eqref{iv}.
\end{proof}

\begin{remark} We observe that in the proof of the uniqueness of the solution it is sufficient to assume only  $$\bs b\in\bs L^1(Q_T)\quad\text{and}\quad c\in L^1(Q_T),$$ instead of \eqref{bc}.

Similarly, we may replace \eqref{coerc} by the different weak coercive assumption by assuming the existence of $r\in\R$, such that, in the sense of distributions,
	\begin{equation*}
	c-\nabla\cdot\bs b\ge r\quad \text{ in }Q_T,
	\end{equation*}
	in order to have also the uniqueness of the solution to the variational inequality \eqref{iv}.
	
	In fact, assuming that there are two solutions $u_1$ and $u_2$, we may choose for test function $v=u_1+\zeta^2s_\zeta(u_2-u_1)$ in the variational inequality
	for $u_1$, where
	$s_\zeta:\R\rightarrow\R$ is a sequence
	of $C^1$ increasing odd functions approximating pointwise the sign function $\sgn^0$ and $\zeta$ is sufficient small. Then, choosing also $v=u_2+\zeta^2s_\zeta(u_1-u_2)$ in the variational inequality
	for $u_2$, we get
	\begin{equation*}
	\int_\Omega\partial_t\bar u(t)\,s_\zeta(\bar u(t))+
	\int_\Omega \bs b(t)\cdot\nabla \bar u(t)\,s_\zeta(\bar u(t))+\int_\Omega c(t)\, \bar u(t)\,s_\zeta(\bar u(t))
	\le 0
	\end{equation*}
	Noting $S_\zeta(s)=\displaystyle\int_0^s s_\zeta(\tau)\,d\tau \underset{\zeta\rightarrow 0}{\longrightarrow} |s|$ and $\tau s_\zeta(\tau) \underset{\zeta\rightarrow 0}{\longrightarrow} |\tau|$, by the dominated convergence theorem, we have
	\begin{equation*}
	\frac{d }{dt}\int_\Omega |\bar u(t)|
	+\int_\Omega \bs b(t)\cdot \nabla |\bar u(t)| +\int_\Omega c(t)\, |\bar u(t)|
	\le 0
	\end{equation*}
	and so
	\begin{equation*}
	\frac{d }{dt}\int_\Omega |\bar u(t)|
	+r\int_\Omega |\bar u(t)|
	\le 0.
	\end{equation*}
	Since $\bar u(0)=0$, by the Gronwall's inequality, we conclude the uniqueness from	
	\begin{equation*}
	\int_\Omega |\bar u(t)|\le e^{rt}\int_\Omega |\bar u(0)|=0.
	\end{equation*}
\end{remark}

\section{Stability and asymptotic behaviour in time}

In this section, the stability of the solutions of the variational inequality \eqref{iv}, as well as its asymptotic limit   when $t\rightarrow+\infty$ is based in the following Lemma, which is due essentially to  \cite{Lisa2002}.

\begin{lemma}\label{mosco}
For $i=1,2$, let $g_i$ belong to $L^\infty(Q_T)$.
If $v_1\in L^q\big(0,T;W^{1,p}_0(\Omega)\big)$, $1\le p,q \le \infty$, is such that $v_1(t)\in\K_{g_1(t)}$ for a.e. $t\in(0,T)$ then there exists $\widehat v_2\in L^q\big(0,T;W^{1,p}_0(\Omega)\big)$
such that $\widehat v_2(t)\in\K_{g_2(t)}$ for a.e. $t\in(0,T)$ and a positive constant $C$ such that
$$\|v_1-\widehat v_2\|_{L^q(0,T;W^{1,p}_0(\Omega))}\le C\|g_1-g_2\|_{L^\infty(Q_T)}.$$
\end{lemma}
\begin{proof} Let $\alpha(t)=\|g_1(t)-g_2(t)\|_{L^\infty(\Omega )}$. Define
$\displaystyle{\psi(t)=1+\frac{\alpha(t)}{m}}$ and $\displaystyle{\widehat v_2(t)=\frac{1}{\psi(t)}\,v_1(t)}$.
	
Since
$$|\nabla\widehat v_2(t)|=\frac{1}{\psi(t)}|\nabla v_1(t)|\le \frac{1}{\psi(t)}g_1(t)$$
and
$$\frac{g_1(t)}{\psi(t)}=\frac{m}{m+\alpha(t)}g_1(t)\le g_2(t)$$
then $\widehat v_2(t)\in\K_{g_2(t)}$ for a.e. $t\in(0,T)$. The conclusion follows immediately from
$$|\nabla(v_1-\widehat v_2)|=\Big|1-\frac{1}{\psi(t)}\Big||\nabla v_1|\le\frac{|\nabla v_1|}{m}\,\|g_1-g_2\|_{L^\infty(Q_T)}.$$
\end{proof}

The continuous dependence result is a consequence of the boundedness of the solution and of its gradient, when we impose the weakly coercive assumption  \eqref{coerc}.

\begin{theorem}\label{dep_cont} For $i=1,2$, let $u_i$ denote the solution of the variational inequality \eqref{iv} with data $(\bs b_i, c_i, f_i,$ $ g_i, u_{0i})$ satisfying assumptions \eqref{bc}-\eqref{g}.
	Then there exists a positive constant $C=C(T)$, depending on $T$, such that
\begin{multline*}
	\|u_1-u_2\|_{L^\infty(0,T; L^2(\Omega))}^2\le C\big(\|u_{01}-u_{02}\|_{L^2(\Omega)}^2+\|\bs b_1-\bs b_2\|_{\bs L^1(Q_T)}+\|c_1-c_2\|_{L^1(Q_T)}\\+\|f_1-f_2\|_{L^1(Q_T)}+\|g_1-g_2\|_{L^\infty(Q_T)}\big).
	\end{multline*}
\end{theorem}
\begin{proof}
Let $\widehat u_2$ be defined as in Lemma \ref{mosco}, for the solution $u_1$ and $\widehat u_1$ be the corresponding function for $u_2$. Using $\widehat u_1$ as test function in the variational inequality \eqref{iv}, we obtain
\begin{multline*}
\int_\Omega\partial_t u_1(t)\big(u_1(t)-u_2(t)\big)+\int_\Omega \bs b_1(t)\cdot\nabla u_1(t)\big(u_1(t)-u_2(t)\big)+\int_\Omega c_1(t)u_1(t)\big(u_1(t)-u_2(t)\big)\\
\le \int_\Omega f_1(t)\big(u_1(t)-u_2(t)\big)+\int_\Omega\big(\partial_t u_1(t)+ \bs b_1(t)\cdot\nabla u_1(t)+c_1(t)u_1(t)-f_1(t)\big)\,\big(\widehat u_1(t)-u_2(t)\big)
\end{multline*}
and a similar inequality is true using the variational inequality of $u_2$, by replacing the data $f_1, \bs b_1, c_1$ by $f_2, \bs b_2, c_2$ and $\widehat u_1$ by $\widehat u_2$. Then we have
\begin{multline}\label{www}
\int_\Omega \partial_t\big(u_1(t)-u_2(t)\big)\,\big(u_1(t)-u_2(t)\big)+\int_\Omega\bs b_1(t)\cdot\nabla\big(u_1(t)-u_2(t)\big)\,\big(u_1(t)-u_2(t)\big)\\+\int_\Omega c_1(t)\,\big(u_1(t)-u_2(t)\big)^2
\le \Theta(t),
\end{multline}
with
\begin{align*}
\Theta(t)&=\int_\Omega\big(\partial_t u_1(t)+ \bs b_1(t)\cdot\nabla u_1(t)+c_1(t)u_1(t)-f_1(t)\big)\,\big(\widehat u_1(t)-u_2(t)\big)\hspace{2cm}\\
&\ \ +\int_\Omega\big(\partial_t u_2(t)+ \bs b(t)\cdot\nabla u_2(t)+c(t)u_2(t)-f_2(t)\big)\,\big(\widehat u_2(t)-u_1(t)\big)\\
&\ \ +\int_\Omega\big(\bs b_1(t)-\bs b_2(t)\big)\cdot\nabla u_2(t)\,\big(u_1(t)-u_2(t)\big)\\
&\ \ + \int_\Omega\big[\big(c_1(t)-c_2(t)\big)\, u_2(t)+\big(f_1(t)-f_2(t)\big)\big]\,\big(u_1(t)-u_2(t)\big).
\end{align*}

Using the boundedness of the solutions $u_i$, $i=1,2,$ and their gradients and recalling the $L^2(Q_T)$ estimates of $\partial_t u_i$, we have
\begin{align*} 
\nonumber	\int_0^T\Theta(\tau)d\tau&\le C_M\big(\|g_1-g_2\|_{L^\infty(Q_T)}+\|\bs b_1-\bs b_2\|_{\bs L^1(Q_T)}+\|c_1-c_2\|_{L^1(Q_T)}+\|f_1-f_2\|_{L^1(Q_T)}\big),
\end{align*}
where $C_M$ is a positive constant depending on $T$, on the norms of the solutions and their derivatives (which can be bounded in terms of the data) and on the constant $C$ of Lemma 3.1.

Setting $w=u_1-u_2$ in the inequality \eqref{www}, we obtain using  \eqref{coerc},

\begin{equation*}
	\frac{d\ }{dt}\int_\Omega |w(t)|^2\le 2|\l|\int_\Omega |w(t)|^2+2\Theta(t).
	\end{equation*}
Applying Gronwall's inequality, we conclude
\begin{multline*}
	\int_\Omega\big|u_1(t)-u_2(t)\big|^2\le e^{2|\l|T}\Big(\|u_{10}-u_{20}\|_{L^2(\Omega)}^2+2C_M\big(\|g_1-g_2\|_{L^\infty(Q_T)} \\+\|\bs b_1-\bs b_2\|_{\bs L^1(Q_T)}+\|c_1-c_2\|_{L^1(Q_T)}+ \|f_1-f_2\|_{L^1(Q_T)}
\big)\Big).
\end{multline*}

\end{proof}

In order to consider the corresponding time independent solution to the first order variational inequality, we give stationary data  $f_\infty,g_\infty,\bs b_\infty,c_\infty$ satisfying the
assumptions
\begin{equation}\label{assumptions_infty1}
g_\infty\in L^\infty(\Omega),\ g_\infty\ge m>0,\qquad f_\infty\in L^1(\Omega),
\end{equation}
\begin{equation}\label{assumptions_infty2}
\bs b_\infty\in \bs L^1(\Omega),\quad c_\infty\in L^1(\Omega),
\end{equation}
\begin{equation}\label{assumption_infty3}
	\ c_\infty-\frac12\nabla\cdot\bs b_\infty\geq \lambda>0\quad\text{in}\quad \Omega
\end{equation}
in the distributional sense,
where we set accordingly
\begin{equation}\label{kginfty}
\K_{g_\infty}=\big\{w\in H^1_0(\Omega ):|\nabla  w|\leq g_\infty\mbox{ a.e. in }\Omega \big\}.
\end{equation}

Then, the stationary problem can be written as
\begin{equation}\label{ivlim}
u_\infty\in\K_{g_\infty}:\quad \int_\Omega \bs b_\infty\cdot\nabla u_\infty (w-u_\infty)+\int_\Omega c_\infty\,u_\infty(w-u_\infty)\geq
\int_\Omega f_\infty(w-u_\infty),\qquad \forall\, w\in\K_{g_\infty}.
\end{equation}

Since the convex set $\K_{g_\infty}$ is bounded in $H^1_0(\Omega )$ and the first order linear operator in the left hand side of \eqref{ivlim} is pseudo-monotone, by the classical theory (see, for instance, \cite{Lions1969}) it has a solution, which is unique by the strict coerciveness induced by the condition $\lambda>0$ in \eqref{assumption_infty3}.

In order to study the asymptotic convergence of the solution of the variational inequality \eqref{iv}
to the stationary solution of \eqref{ivlim}, we consider solutions global in time. This is easily obtained if we assume that \eqref{bc}-\eqref{dados} are satisfied for any $T>0$ and replace \eqref{g} by
\begin{align}\label{asinfty}
g\in W^{1,\infty}\big(0,\infty;L^\infty(\Omega)\big),\  g\ge m>0.
\end{align}

We need an auxiliary lemma.
\begin{lemma}
\label{anterior}
{\em (\cite{Haraux1981}, pg. 286)} Let $\varphi:(0,\infty)\rightarrow\R$ be a nonnegative
function, absolutely continuous in any compact subinterval of
$(0,\infty)$, $\Phi\in L^1_{loc}(0,\infty)$ a nonnegative function and
$\mu$ a positive constant, such that,
\begin{equation}
\label{zeta}
\varphi'(t)+\mu\varphi(t)\le\Phi(t),\qquad \qquad \forall\,\ t>0.
\end{equation}

Then, for any $s,t>0$,
\begin{equation*}
\varphi(t+s)\le e^{-\mu t}+\frac 1{1-e^{-\mu}}\left[\sup_{\tau\ge s}\int_\tau^{\tau+1}
\Phi(\xi)d\xi\right].
\end{equation*}
\hfill{$\square$}
\end{lemma}

In order to apply this Lemma to
\begin{equation}\label{zeta2}
\varphi(t)= \int_\Omega\big|u(t)-u_\infty\big|^2, \qquad\ t>0,
\end{equation}
we shall require the additional assumptions on the coefficients and on the data
\begin{equation}\label{bcf}
\bs b\in L^\infty \big(0,\infty;\bs L^2(\Omega)\big) \quad\text{and}\quad c, f \in L^\infty\big(0,\infty; L^2(\Omega)\big).
\end{equation}

\begin{theorem} Assume that $f,g,\bs b, c, u_0$ satisfy the assumptions  \eqref{bc}-\eqref{dados}, \eqref{asinfty}, \eqref{bcf}
and $f_\infty, g_\infty,\bs b_\infty, c_\infty$  satisfy the assumption \eqref{assumptions_infty1},\eqref{assumptions_infty2} and \eqref{assumption_infty3}.
Suppose, in addition, that
\begin{equation*}
\int_t^{t+1}\int_\Omega \big|f(\tau)-f_\infty\big|d\tau dx\underset{t\rightarrow\infty}{\longrightarrow} 0,\quad
\int_t^{t+1}\int_\Omega \big|\bs b(\tau)-\bs b_\infty\big|d\tau dx\underset{t\rightarrow\infty}{\longrightarrow} 0,\quad
\int_t^{t+1}\int_\Omega \big(c(\tau)-c_\infty\big)d\tau dx\underset{t\rightarrow\infty}{\longrightarrow} 0
\end{equation*}
and there exists $\gamma>\frac 12$, such that, for some constant $D>0$,
\begin{equation}\label{conditi2}
 \|g(t)-g_\infty\|_{L^\infty(\Omega )}\le\frac D{t^\gamma}, \qquad t>0.
\end{equation}

If $u$ and $u_\infty$ are, respectively, the unique solutions of the variational inequalities \eqref{iv} and \eqref{ivlim} then, for every $\alpha$, $0<\alpha<1$,
\begin{equation*}
u(t)\underset{t\rightarrow\infty}{\longrightarrow} u_\infty\quad\mbox{ in }\C^{0,\alpha}(\overline{\Omega })
\end{equation*}
\end{theorem}
\begin{proof}
First we need to return to the estimate \eqref{te} of the existence proof in order to prove that, under the additional assumptions of this theorem, there are positive constants $A,B$, independent of $T$, such that,
\begin{equation*}
\label{AB}\|\partial_t u\|_{L^2(\Omega\times(0,T))}\le A\sqrt T+B.
\end{equation*}
Since $|\nabla  u(x,t)|\le g(x,t)$ for a.e. $(x,t)\in Q_\infty=\Omega\times(0,\infty)$ and $g\in L^\infty(Q_\infty)$, we have now $u\in L^\infty(0,\infty;W^{1,\infty}(\Omega ))$. This yields the estimate
$$\|u\|^2_{L^2(Q_T)}=\int_{Q_T}|u|^2\le c_g T$$
where the constant $c_g>0$ is independent of $T$. Using similar estimates for $\|f\|^2_{L^2(Q_T)}$ with the constant $c_g$ replaced by $c_f=\|f\|^2_{L^2(0,\infty;L^2(\Omega))}$, as well as for $c_{\bs b}=\|\bs b\|^2_{\bs L^2(0,\infty;L^2(\Omega))}$  and $c_c=\|c\|^2_{L^2(0,\infty;L^2(\Omega))}$, we may conclude that the constant $C_1=C_1(T)$ of  \eqref{k}, in the Estimate 1, grows also linearly with $T$, i.e. $C_1\le c_0+c_1 T$, where $c_0$ depends only on $u_0$ and $c_1$ depends on $m$, $c_f$, $c_g$, $c_{\bs b}$ and $c_c$. Using this fact in the Estimate 3, we may now easily deduce \eqref{AB} from \eqref{te}, with $s=2$ and $q=\infty$, since $C_4$, depending on $f$ and on $C_1$ grows also linearly with $T$.

Using Lemma \ref{mosco}, we choose $\widehat u_\infty\in\K_{g(t)}$, for a.e. $t\in(0,T)$, as test function in \eqref{iv}. Then
\begin{multline*}
	\int_\Omega\partial_t u(t)\big(u(t)-u_\infty\big)+\int_\Omega \bs b(t)\cdot\nabla u(t)\big(u(t)-u_\infty\big)+\int_\Omega c(t)u(t)\big(u(t)-u_\infty\big)\\
	\le \int_\Omega f(t)\big(u(t)-u_\infty\big)+\int_\Omega\big(\partial_t u(t)+ \bs b(t)\cdot\nabla u(t)+c(t)u(t)-f(t)\big)\,\big(\widehat u_\infty-u_\infty\big).
\end{multline*}

Analogously, with $\widehat u(t)\in\K_{g_\infty}$, for a.e. $t\in(0,T)$, we obtain the inequality
\begin{multline*}
\int_\Omega \bs b_\infty\cdot\nabla u_\infty\big(u(t)-u_\infty\big)+\int_\Omega c_\infty u_\infty\big(u(t)-u_\infty\big)\\
	\ge \int_\Omega f_\infty\big(u(t)-u_\infty\big)+\int_\Omega\big(\bs b_\infty\cdot\nabla u_\infty+c_\infty u_\infty-f_\infty\big)\,\big(u(t)-\widehat u(t)\big).
\end{multline*}
Then, simple algebraic manipulations lead to
\begin{multline}\label{asas}
	\int_\Omega\partial_t \big(u(t)-u_\infty\big)\big(u(t)-u_\infty\big)+\int_\Omega \bs b_\infty \cdot\nabla \big(u(t)-u_\infty\big)\big(u(t)-u_\infty\big)+\int_\Omega c_\infty \big(u(t)-u_\infty\big)\big(u(t)-u_\infty\big)\\
	\le\Theta(t),
\end{multline}
where
\begin{multline*}
	\Theta(t)=\int_\Omega\big(\partial_t u(t)+ \bs b(t)\cdot\nabla u(t)+c(t)u(t)-f(t)\big)\,\big(\widehat u_\infty-u_\infty\big)+\int_\Omega\big(\bs b_\infty\cdot\nabla u_\infty+c_\infty u_\infty-f_\infty\big)\,\big(\widehat u(t)-u(t)\big) \\
+\int_\Omega(\bs b(t)-\bs b_\infty)\cdot\nabla u(t)\big(u_\infty-u(t)\big)+\int_\Omega(c(t)-c_\infty)u(t)\big(u_\infty-u(t)\big)+\int_\Omega \big(f(t)-f_\infty \big)\big(u(t)-u_\infty\big).
\end{multline*}

Using \eqref{assumption_infty3} and the definition \eqref{zeta2}, from \eqref{asas}, we obtain the differential inequality with $\mu=2\lambda$ and where, taking into account \eqref{AB}, we may choose $\Phi(t)\ge2|\Theta(t)|$ given by
\begin{multline*}
	\Phi(t)=C\big((A\sqrt t+B+C)\|g(t)-g_\infty\|_{L^\infty(\Omega)}+\|\bs b(t)-\bs b_\infty\|_{\bs L^1(\Omega)}+\|c(t)-c_\infty\|_{L^1(\Omega)}+\|f(t)-f_\infty\|_{L^1(\Omega)}\big).
\end{multline*}

Then, using the assumptions and observing that the number $\gamma$ in \eqref{conditi2} is greater than $\frac12$, we have
\begin{multline*}
\int_t^{t+1}\Phi(\tau)d\tau\le C\int_t^{t+1}\big(\|f(\tau)-f_\infty\|_{L^1(\Omega)}+\|\bs b(\tau)-\bs b_\infty\|_{L^1(\Omega)}+\|c(\tau)-c_\infty\|_{L^1(\Omega)}\big)d\tau\\
 +C'\int_t^{t+1}(\tau^\frac12+1)\|g(\tau)-g_\infty\|_{L^\infty(\Omega)}d\tau
\underset{t\rightarrow+\infty}{\longrightarrow}0.
\end{multline*}
Therefore, by Lemma \ref{anterior}, $u(t)\underset{t\rightarrow+\infty}{\longrightarrow}u_\infty$ in $L^2(\Omega)$. Since $u$ belongs to $L^\infty\big(0,\infty;W^{1,\infty}(\Omega)\big)$, the compact inclusion of $W^{1,\infty}(\Omega)$ in $C^{0,\alpha}(\overline\Omega)$ implies, first for a subsequence, and after for the whole sequence, that $u(t)\underset{t\rightarrow+\infty}{\longrightarrow}u_\infty$ in $C^{0,\alpha}(\overline\Omega)$, concluding the proof.

\end{proof}

\section{Finite time stabilization in a special case}

In this section we assume that $\partial\Omega$ is of class $\C^2$ and
\begin{equation}\label{sc}
\bs b\in\R^N,\quad c\equiv0,\quad g\equiv1,\quad z_0\in\K_1\quad \text{and}\quad f\in L^\infty(0,T).
\end{equation}

We consider the following two obstacles problem
\begin{equation}
\label{ivob}
\begin{array}{l}
\left\{\begin{array}{l}
z(t)\in\K_\vee^\wedge \mbox{ for a.e. }t\in (0,T),\qquad  z(0)=u_0,\\
\\
\displaystyle\int_\Omega  \partial_t z(t)(v-z(t))+\displaystyle\int_\Omega\bs b\cdot\nabla z(t)(v-z(t))
\geq\displaystyle\int_\Omega  f(t)(v-z(t)),
\ \forall \, v\in\K_\vee^\wedge,\mbox{ for a.e. }t\in (0,T),
\end{array}
\right.
\end{array}
\end{equation}
where
\begin{equation*}
\K_\vee^\wedge=\{v\in H^1_0(\Omega):- d(x)\le v(x)\leq  d(x) \text{ for a.e. }x\in\Omega\}.
\end{equation*}
Here $d(x)=d(x,\partial\Omega)$ is the distance function to the boundary $\partial\Omega$. Notice that $d\in W^{1,\infty}_0(\Omega)$, $|\nabla d(x)|\le1$, a.e. $x\in\Omega$ and $\Delta d\le C$ for some constant $C=C(\Omega)>0$. Observe that $z_0\in\K_1\subset\K_\vee^\wedge$.

\begin{theorem}
Under the assumptions \eqref{sc}, the inequality \eqref{ivob} has a unique solution $$z\in L^\infty\big(0,T;W^{1,\infty}_0(\Omega)\big)\cap H^1\big(0,T;L^2(\Omega)\big)\cap\C(\overline Q_T),$$ which satisfies $|\nabla z|\le1$ a.e. in $Q_T$ and is the unique solution of the variational inequality \eqref{iv}.
\end{theorem}
\begin{proof} For $\eps,\delta\in(0,1)$, we consider the following family of penalized
problems for $z^{\eps\delta}$,
\begin{equation}\label{apob}
\left\{\begin{array}{l}
\partial_tz^{\eps\delta}-\delta\Delta z^{\eps\delta}+\bs b\cdot\nabla z^{\eps\delta}+
\frac\delta\eps\big(z^{\eps\delta}-(z^{\eps\delta}\wedge d)\vee (- d)\big)=f^\delta\ \text{in }Q_T,\vspace{3mm}\\
z^{\eps\delta}(0)=z_0^\eps\ \text{on }\Omega,\qquad z^{\eps\delta}=0\ \text{on }\partial\Omega\times(0,T),
\end{array}
\right.
\end{equation}
where $f^\delta$ and $z_0^\eps$ are regularizations of the functions $f$ and $z_0$, with $|\nabla {z_0^\eps}|\le 1$. This problem has a unique solution $z^{\eps\delta}\in H^{2,1}(Q_T)$,
since the operator
\begin{equation}\label{b}
\langle P_\eps v,w\rangle=\frac\delta\eps\int_\Omega \big(v-(v\wedge  d)\vee(- d)\big)w
\end{equation}
is monotone (see, for instance, \cite{Lions1969}).

We obtain firstly an estimate of  $|\nabla z^{\eps\delta}|$ on $\partial\Omega\times(0,T)$. Since $\partial
\Omega$ is of class $\C^2$, there exists $r>0$ such that, if $B_r(x)$ denotes the ball with centre in $x$ and radius $r$, then for all $x_0\in\partial\Omega$ there exists $y_0\in\R^N$ such that
$\overline B_r(y_0)\cap\overline\Omega=\{x_0\}$. Placing the origin of the coordinates in the point $y_0$, let  $\eta_\eps(s)=e^{-\frac{s}{\sqrt\eps}}$
and
$$\overline\varphi(x)= d(x)+M\eps\big(1-\eta_\eps(|x|-r)\big),\qquad \underline\varphi(x)=- d(x)-M\eps\big(1-\eta_\eps(|x|-r)\big),$$
where $M$ is a positive constant, depending on $\delta $, to be chosen later. We show that $\overline\varphi$ is a supersolution
 of \eqref{apob}. Analogously, it can be verified that $\underline\varphi$  is a subsolution.  We start by observing that
$$\overline\varphi(x_0)=0=z^{\eps\delta}(x_0,t)\qquad \ \text{and} \qquad\overline\varphi\ge 0=z^{\eps\delta} \ \text{on }\partial\Omega\times(0,T).$$

Since $z_0^\eps\in\K_\vee^\wedge$, then
$$\overline\varphi(x)\ge d(x)\ge z_0^\eps(x).$$

We compute
\begin{equation}\label{gdfi}
\partial_{x_i}\overline\varphi(x)=\partial_{x_i}d(x)+M\sqrt\eps\eta_\eps(|x|-r)\tfrac{x_i}{|x|}\quad\text{and}\quad|\nabla\overline\varphi|\le1+M\sqrt\eps,
\end{equation}
$$\partial^2_{x_i}\overline\varphi(x)=\partial^2_{x_i}d(x)-M\eta_\eps(|x|-r)\,\tfrac{x_i^2}{|x|^2}+M\sqrt\eps\eta_\eps(|x|-r)
\Big(\tfrac1{|x|}-\tfrac{x_i^2}{|x|^3}\Big)$$
and
$$\Delta\overline\varphi(x)=\Delta d(x)+M\eta_\eps(|x|-r)\Big(-1+\sqrt\eps\,\tfrac{N-1}{|x|}\Big).$$

Let $$Lw=\partial_tw-\delta\Delta w+\bs b\cdot\nabla w+\frac\delta\eps(w-(w\wedge d)\vee(- d)).$$ 
Then, recalling that there exists a positive constant $C$ such that $\Delta d\le C$ and choosing $\eps$ sufficiently small, such that, $1-\sqrt\eps\,\tfrac{N-1}{|x|}\ge1-\sqrt\eps\,\frac{N-1}{r}\ge\frac12$ we have
\begin{align}\label{sobre}
\nonumber L\overline\varphi-f&=-\delta\Delta d+M\,\delta\,\eta_\eps(|x|-r)\Big(1-\sqrt\eps\,\tfrac{N-1}{|x|}\Big)+
\bs b\cdot\Big(\nabla d+M\sqrt\eps\,\eta_\eps(|x|-r)\,\tfrac{x}{|x|}\Big)+M\,\delta\,\big(1-\eta_\eps(|x|-r)\big)-f\\
\nonumber&\ge -\delta\,C+M\tfrac\delta{2}\,\eta_\eps(|x|-R)-|\bs b|-|\bs b|\,M\,\sqrt\eps\,\eta_\eps(|x|-R)+M\,\delta
\big(1-\eta_\eps(|x|-r)\big)-\|f\|_{L^\infty(0,T)}\\
&=M\left(\delta+(\tfrac\delta{2}-|\bs b|\,\sqrt\eps-\delta)\,\eta_\eps(|x|-r)\right)-\delta\,C-|\bs b|-\|f\|_{L^\infty(0,T)}.
\end{align}

 Observe now that the term $\tfrac\delta{2}-|\bs b|\,\sqrt\eps-\delta$ is negative and, since $\eta_\eps(|x|-r)\le 1$, we have the following inequality
$$M\left(\delta+(\tfrac\delta{2}-|\bs b|\,\sqrt\eps-\delta)\,\eta_\eps(|x|-r)\right)\ge M\big(\tfrac\delta{2}-|\bs b|\,\sqrt\eps\big).$$
We can fix $\eps_0$ such that, for $0<\eps\le\eps_0$, we have $|\bs b|\,\sqrt\eps\le\frac\delta{4}$. From \eqref{sobre}, we obtain then
$$L\overline\varphi-f\ge M\,\tfrac\delta{4}-\delta\,C-|\bs b|-\|f\|_{L^\infty(0,T)}=0,$$
provided
\begin{equation}\label{mmm}
	M=\frac{C_1}\delta,\qquad C_1=4(\delta\,C+|\bs b|+\|f\|_{L^\infty(0,T)}),
\end{equation}
concluding then that $\overline\varphi$ is a supersolution of \eqref{apob}.
 Analogously,  $\underline\varphi$  is a subsolution of \eqref{apob} and so we have
\begin{equation}\label{pkapa}
\underline\varphi\le z^{\eps\delta}\le\overline\varphi\quad \text{in}\,\, Q_T\qquad \text{and}\qquad z^{\eps\delta}(x_0,t)=\overline\varphi(x_0)=\underline\varphi(x_0).
\end{equation}

Observe that, from \eqref{gdfi}, we obtain
$$|\nabla z^{\eps\delta}(x_0,t)|\le\max\{|\nabla \overline\varphi(x_0)|,|\nabla \underline\varphi(x_0)|\}\leq 1+\tfrac{C_1}\delta\,\sqrt\eps$$
for an arbitrary point $x_0\in \partial\Omega$ at any $t\in(0,T)$.
We wish to prove that this estimate is true a.e. in $Q_T$.
Differentiate the first equation of \eqref{apob} with respect to $x_k$, multiply it by $z^{\eps\delta}_{x_k}$ and sum over $k$. Setting
$v=|\nabla z^{\eps\delta}|^2$ and noticing that $z^{\eps\delta}_{x_k}\Delta z^{\eps\delta}_{x_k}=\frac12\Delta v-(z^{\eps\delta}_{x_kx_k})^2$ we get
$$\frac12\partial_t v-\frac\delta2\Delta v+\frac12\bs b\cdot\nabla v+\frac\delta\eps\big(v-\nabla\tilde z^{\eps\delta}\cdot\nabla z^{\eps\delta}\big)\le0,$$
being $\tilde z^{\eps\delta}= z^{\eps\delta}-(z^{\eps\delta}\wedge d)\vee (-d)$. Using the Cauchy-Schwartz inequality, we obtain
$$\partial_t v-\delta\Delta v+\bs b\cdot\nabla v+\frac{2\delta}\eps\big(v-|\nabla\tilde z^{\eps\delta}| v^\frac12\big)\le 0.$$
Multiplying the above inequality by $(v-(1+M\sqrt\eps)^2)^+$ and integrating over $Q_t$, we have
\begin{multline}\label{meps}
	\frac12\int_\Omega\big|(v(t)-(1+M\sqrt\eps)^2)^+\big|^2+\delta\int_{Q_t}|\nabla (v-(1+M\sqrt\eps)^2)^+|^2\\
	+\int_{Q_t}\bs b\cdot\nabla (v-(1+M\sqrt\eps)^2)^+\,
	(v-(1+M\sqrt\eps)^2)^+
	+\frac{2\delta}\eps
	\int_{Q_t}\big(v-|\nabla\tilde z^{\eps\delta}| v^\frac12\big)(v-(1+M\sqrt\eps)^2)^+\le 0.
\end{multline}

Since $$\displaystyle \int_{Q_t}\bs b\cdot\nabla (v-(1+M\sqrt\eps)^2)^+\,(v-(1+M\sqrt\eps)^2)^+=0$$ and
\begin{multline*}
	\int_{Q_t}\big(v-|\nabla\tilde z^{\eps\delta}| v^\frac12\big)(v-(1+M\sqrt\eps)^2)^+\\
	=\int_{\{z^{\eps\delta}>d\}}
	\big(v- v^\frac12\big)(v-(1+M\sqrt\eps)^2)^++\int_{\{z^{\eps\delta}<-d\}}
	\big(v- v^\frac12\big)(v-(1+M\sqrt\eps)^2)^+\ge0,
\end{multline*}
from \eqref{meps} we conclude that $(v-(1+M\sqrt\eps)^2)^+\equiv0$.

Then, recalling the choice of $M$ done in \eqref{mmm}, we have
\begin{equation}\label{estgd2ob}
	|\nabla z^{\eps\delta}|^2=v\le 1+\tfrac{C_1}\delta\sqrt\eps\quad\text{ a.e. in }Q_T,
\end{equation}
and $\{z^{\eps\delta}\}_{\eps}$ is uniformly bounded in $L^\infty\big(0,T;W^{1,\infty}_0(\Omega)\big) $. Using \eqref{pkapa}, it is easy see that
$$-C_1\le\frac\delta\eps(z^{\eps\delta}-(z^{\eps\delta}\wedge d)\vee(-d))\le C_1.$$
In fact, in the set $\{z^{\eps\delta}>d\}$ we have
$$\frac\delta\eps(z^{\eps\delta}-(z^{\eps\delta}\wedge d)\vee(-d))=\frac\delta\eps(z^{\eps\delta}-d)\le\frac\delta\eps(\overline\varphi-d)\le C_1,$$
in the set $\{-d\le z^{\eps\delta}\le d\}$ we have $z^{\eps\delta}-(z^{\eps\delta}\wedge d)\vee(-d)=0$ and in the set $\{z^{\eps\delta}<-d\}$ we have
$$\frac\delta\eps(z^{\eps\delta}-(z^{\eps\delta}\wedge d)\vee(-d))=\frac\delta\eps(z^{\eps\delta}+d)\ge
\frac\delta\eps(\underline\varphi+d) \ge-C_1.$$

Multiplying the first equation of \eqref{apob} by $\partial_tz^{\eps\delta}$, we obtain
\begin{equation*}
		\int_{Q_t}|\partial_tz^{\eps\delta}|^2+\delta\int_{Q_t} \nabla z^{\eps\delta}\cdot\nabla\partial_t z^{\eps\delta}+
		\int_{Q_T}\bs b\cdot\nabla z^{\eps\delta}\partial_t z^{\eps\delta}
		+	\frac\delta\eps\int_{Q_t}\big(z^{\eps\delta}-(z^{\eps\delta}\wedge d)\vee(-d)\big)\partial_tz^{\eps\delta}=\int_{Q_t}f\partial_tz^{\eps\delta}
	\end{equation*}
	and so
\begin{align*}
	\int_{Q_t}|\partial_tz^{\eps\delta}|^2&+\frac\delta{2}\int_\Omega|\nabla z^{\eps\delta}(t)|^2\\
	&\le\frac\delta{2}\int_\Omega|\nabla u^{\eps}_0|^2 +\big(|\bs b|\,\|\nabla z^{\eps\delta}\|_{\bs L^2(Q_T)}+\|\frac\delta\eps(z^{\eps\delta}-(z^{\eps\delta}\wedge d)\vee(-d)\|_{L^\infty(Q_T)}\\
	&\quad\quad
	+\|f\|_{L^2(0,T)}\big)\|\partial_tz^{\eps\delta}\|_{L^2(Q_T)}\\
	&\le\frac\delta{2}\int_\Omega|\nabla u_0^\eps|^2 + \big(|\bs b|(1+\tfrac{C_1}\delta\,\sqrt\eps+C_1)|Q_T|^\frac12+\|f\|_{L^2(0,T)}\big)\|\partial_tz^{\eps\delta}\|_{L^2(Q_T)},\\
	&\le \frac\delta{2}\int_\Omega|\nabla u_0^\eps|^2+\frac12\big(|\bs b|(1+\tfrac{C_1}\delta\,\sqrt\eps+C_1)|Q_T|^\frac12+\|f\|_{L^2(0,T)}\big)^2+\frac12\|\partial_tz^{\eps\delta}\|_{L^2(Q_T)}^2
\end{align*}
where $|Q_T|$ denotes the Lebesgue measure of $Q_T$.
So, for $\delta$ fixed,
\begin{equation}\label{zt}
\|\partial_tz^{\eps\delta}\|_{L^2(Q_T)}^2\le  \delta\int_\Omega|\nabla u_0^\eps|^2+\big(|\bs b|(1+\tfrac{C_1}\delta\,\sqrt\eps+C_1)|Q_T|^\frac12+\|f\|_{L^2(0,T)}\big)^2.
	\end{equation}
Then, there exists $z^\delta\in L^\infty\big(0,T;W^{1,\infty}_0(\Omega)\big)\cap H^1\big(0,T;L^2(\Omega)\big)$ such that
$$z^{\eps\delta}\underset{\eps\rightarrow0}{\lraup }z^\delta\text{ in }L^\infty\big(0,T;W^{1,\infty}_0(\Omega)\big)\text{-weak}*\quad\text{and}\quad
\partial_tz^{\eps\delta}\underset{\eps\rightarrow0}{\lraup }\partial_tz^\delta\text{ in }L^2(Q_T).$$
 Multiplying the first equation of the problem \eqref{apob} by $v-z^{\eps\delta}(t)$, where $v\in\K_\vee^\wedge$ and integrating over
$\Omega\times(s,t)$, $0\leq s<t\le T$,  we obtain
\begin{multline*}
\int_s^t\int_\Omega \partial_t z^{\eps\delta}(v-z^{\eps\delta})+\delta\int_s^t\int_\Omega\nabla z^{\eps\delta}\cdot\nabla(v-z^{\eps\delta})+
\int_s^t\int_\Omega\bs b\cdot\nabla z^{\eps\delta}\,(v-z^{\eps\delta})\\
+\frac\delta\eps\int_\Omega\big(z^{\eps\delta}-(z^{\eps\delta}(t)\wedge  d)\vee(- d)\big)
(v-z^{\eps\delta})=\int_s^t\int_\Omega f^\delta(v-z^{\eps\delta}).
\end{multline*}
For $v\in\K_\vee^\wedge$, the operator $P_\eps$ defined in \eqref{b} is monotone, we have
$$\int_s^t\int_\Omega\big(z^{\eps\delta}-(z^{\eps\delta}\wedge d)\vee(- d)\big)
(v-z^{\eps\delta})\le0.$$
So, letting $\eps\rightarrow0$, we obtain
\begin{equation}\label{ivdeltad}
\int_s^t\int_\Omega \partial_t z^{\delta}(v-z^{\delta})+\delta\int_s^t\int_\Omega\nabla z^{\delta}\cdot(v-z^{\delta})+
\int_s^t\int_\Omega\bs b\cdot\nabla z^{\delta}\,(v-z^{\delta})\ge\int_s^t\int_\Omega f^\delta(v-z^{\delta}).
\end{equation}
By \eqref{pkapa}, the function $z^\delta$ is such that $z^\delta(t)\in\K_\vee^\wedge$, for a.e. $t\in(0,T)$.

To prove that $\{\partial_t z^\delta\}_\delta$ is bounded in $L^2(Q_T)$, let $\eps\rightarrow0$ in \eqref{zt}, obtaining
$$\int_{Q_t}|\partial_t z^\delta|^2 \le\delta\int_\Omega|\nabla u_0|^2 + \big(|\bs b|(1+C_1)|Q_T|^\frac12+\|f\|_{L^2(0,T)}\big)^2.$$
Analogously, letting $\eps\rightarrow0$ in \eqref{estgd2ob}, we obtain
$$|\nabla z^{\delta}|\le 1\quad\text{ a.e. in }Q_T.$$

We can now pass easily to the limit when $\delta\rightarrow0$ in inequality \eqref{ivdeltad}. Observing that $z^\delta$ converges to
some function $z$ weakly* in $L^\infty\big(0,T;W^{1,\infty}_0(\Omega)\big)$ and $ \partial_t z^{\delta}$ converges weakly in $L^2(Q_T)$
to $\partial_t z$, we find for all $0<s<t<T$
$$
\int_s^t\int_\Omega \partial_t z(v-z)+
\int_s^t\int_\Omega\bs b\cdot\nabla z\,(v-z)\ge\int_s^t\int_\Omega f(v-z)$$
and it follows also
 $$
\int_\Omega \partial_t z(t)(v-z(t))+
\int_\Omega\bs b\cdot\nabla z(t)\,(v-z(t))\ge\int_\Omega f(t)(v-z(t))\,\,\,\, \text{for a.e.}\,\, t\in(0,T).$$

Since $z^\delta(t)\in\K_\vee^\wedge$ for a.e. $t\in(0,T)$, we also have $z(t)\in\K_\vee^\wedge$ and the proof of existence of solution for the variational inequality \eqref{ivob} is complete. The uniqueness is also clear.

The inclusion $\K_1\subset\K_\vee^\wedge$ and the fact $z(t)\in\K_1$ for a.e. $t\in(0,T)$ implies that the function $z$ also solves the problem \eqref{iv}.
\end{proof}

\begin{remark}
	The first order variational inequalities of obstacle type have been introduced by Bensoussan and Lions in \cite{BensoussanLions1973} and have been studied in \cite{MignotPuel1976} and in \cite{Rodrigues1},  for general linear operators and general obstacles, and extended to a quasilinear two obstacles problem in \cite{LeviVallet2001}. In all those cases the notion of solution is less regular and the boundary data can only be prescribed on part of the boundary. In addition, the solution cannot have a gradient in $L^2$ and the best that can be expected in general is the operator $\partial_t u+\bs b\cdot\nabla u+c\,u \in L^2$, as a consequence of Lewy-Stampacchia inequalities. These estimates can be obtained from the regularized parabolic inequality \eqref{ivdeltad} and, as in \cite{Rodrigues2}, it allows the passage to the limit  $\delta\rightarrow0$ without the estimates on the gradient and on the time derivative. It is an open question to establish the equivalence of the first order obstacle problem with the variational inequality with gradient constraint for more general first order linear operators.
	\end{remark}

\begin{theorem} In addition to the assumptions \eqref{sc}, suppose
\begin{equation}\label{ha}
\bs b\cdot\nabla z_0\le f(t)\text{ in }\big\{x\in\Omega:-d(x)<z_0(x)\big\} \,\,for\,\, t>0,
\end{equation}
\begin{equation}\label{ff}
	 f=f(t)\,\,\text{is increasing and nonnegative},
	 \end{equation}
	 \begin{equation}\label{liminf}
\liminf_{t\rightarrow\infty}f(t)>|\bs b|+2D,
\end{equation}
where $D=\|d\|_{L^\infty (\Omega)}=\displaystyle\max_{x\in\overline\Omega}d(x,\partial\Omega)$.
Then there exists $T_*<\infty$ such that the solution $z$ of the variational inequality \eqref{iv}, or equivalently of  \eqref{ivob}, satisfies
$$z(t)=d \quad\text{for all} \,\,t\ge T_*.$$
\end{theorem}
\begin{proof} We consider $z$ as the solution of the variational inequality \eqref{ivob}.

\vspace{3mm}

\noindent{\bf Step 1:} $z_0\le z(t)\quad\text{ for all } t>0.$

\vspace{3mm}

Let $v(t)=z(t)+(z_0-z(t))^+$ and note that $v(t)\in\K_\vee^\wedge$. Then
\begin{equation}\label{utu0}
\int_\Omega\partial_tz(t)(z_0-u(t))^++\int_\Omega\bs b\cdot\nabla z(t)(z_0-z(t))^+\ge\int_\Omega f(t)(z_0-z(t))^+.
\end{equation}
On the other hand, by \eqref{ha}, we have
\begin{equation*}
\bs b\cdot\nabla z_0(z_0-z(t))^+\leq f(t)(z_0-z(t))^+\quad\text{ in }\quad\{-d<z_0\},
\end{equation*}
and also on $\{-d=z_0\}$ since, in this last set, $(z_0-z(t))^+\equiv 0$ (recall that $f\ge 0$). Then
\begin{equation}\label{u0ut}
\int_\Omega\partial_tz_0(z_0-z(t))^++\int_\Omega\bs b\cdot\nabla z_0(z_0-z(t))^+\le\int_\Omega f(t)(z_0-z(t))^+.
\end{equation}

From \eqref{utu0} and \eqref{u0ut} we get
\begin{equation*}
\int_\Omega\partial_t(u_0-z(t))(z_0-z(t))^++\int_\Omega\bs b\cdot\nabla(z_0- z(t))(z_0-z(t))^+\le 0.
\end{equation*}

But
\begin{equation*}
\int_\Omega\bs b\cdot\nabla(z_0- z(t))(z_0-z(t))^+=\frac12\int_\Omega\bs b\cdot\nabla\big((z_0-z(t))^+\big)^2=-\frac12
\int_\Omega\nabla\cdot\bs b\,\big((z_0-z(t))^+\big)^2=0
\end{equation*}
and so
\begin{equation*}
\frac12\int_\Omega \big|(z_0-z(t))^+\big|^2\leq \frac12\int_\Omega \big|(z_0-z(0))^+\big|^2=0,
\end{equation*}
which implies that $z_0\leq z(t)$, for all $t\>0$.

\vspace{3mm}

\noindent{\bf Step 2:} $z(t)\le z(t+h)\quad\text{ for all }t,h>0.$

\vspace{3mm}

Observe that  $v(t)=z(t+h)-(z(t)-z(t+h))^-\in\K_\vee^\wedge$, so we can choose $v(t)$  as test function in \eqref{ivob}. Noting that
$$v(t)-z(t)=z(t+h)-z(t)-\big(z(t)-z(t+h)\big)^-=-\big(z(t)-z(t+h)\big)^+$$
we get
\begin{equation}\label{ut-h}
-\int_\Omega\partial_tz(t)\big(z(t)-z(t+h)\big)^+-\int_\Omega\bs b\cdot\nabla z(t)\big(z(t)-z(t+h)\big)^+\ge-\int_\Omega f(t)\big(z(t)-z(t+h)\big)^+.
\end{equation}

Choosing $v(t)=z(t+h)+\big(z(t)-z(t+h)\big)^+$ as test function in \eqref{ivob} in the instant $t+h$ and observing that
$$v(t)-z(t)=z(t+h)-z(t)+\big(z(t)-z(t+h)\big)^+=\big(z(t)-z(t+h)\big)^-,$$
we have
\begin{equation}\label{uth}
\int_\Omega\partial_tz(t+h)\big(z(t)-z(t+h)\big)^-+\int_\Omega\bs b\cdot\nabla z(t+h)\big(z(t)-z(t+h)\big)^-\ge\int_\Omega f(t+h)\big(z(t)-z(t+h)\big)^-.
\end{equation}

From \eqref{ut-h} and \eqref{uth} we get
\begin{multline*}
\int_\Omega\partial_t(z(t)-z(t+h))\big(z(t)-z(t+h)\big)^-
+\int_\Omega\bs b\cdot\nabla(z(t)- z(t+h))\,\big(z(t)-z(t+h)\big)^-\\
\le\int_\Omega (f(t)-f(t+h))\big(z(t)-z(t+h)\big)^-\le0,
\end{multline*}
because $f(t)\leq f(t+h)$, by assumption \eqref{ff} and $\big(z(t)-z(t+h)\big)^-\ge0$. As
$$\int_\Omega\bs b\cdot\nabla(z(t)- z(t+h))\,\big(z(t)-z(t+h)\big)^-=0,$$
we obtain
$$\frac12\int_\Omega\Big(\big|(z(t)-z(t+h)\big)^-\Big|^2\le\frac12\int_\Omega\Big(\big|(z(0)-z(h)\big)^-\Big|^2\le0,$$
using Step 1. So $z(t)\leq z(t+h)$, for all $t,h>0$.

\vspace{3mm}

\noindent{\bf Step 3:} There exists $z_\infty\in\C(\overline\Omega)$ such that $\displaystyle\lim_{t\rightarrow+\infty}z(x,t)=z_\infty(x),$
uniformly in $x\in\overline\Omega$.

\vspace{3mm}
Since the sequence of continuous functions ${\{z(t)\}}_{t >0}$ is increasing in $t$ and is bounded from above by $d$, this conclusion follows immediatly.

\vspace{3mm}

However, in this special case we have a finite time stabilization.

First we prove that the function $z_\infty$ coincides with $d$. We recall that $\partial_tz\in L^2(Q_T)$, for any $T>0$,
and we set $\psi(t)=\displaystyle\int_\Omega z(t)$. Observe that $\|\psi\|_{L^\infty(0,\infty)}\leq |\Omega| D$, where $|\Omega|$ denotes the Lebesgue measure of $\Omega$. Since ${\{z(t)\}}_{t >0}$ is
increasing, then $\partial_tz\ge 0$ and
\begin{equation*}
	\psi(t)\underset{t\rightarrow+\infty}{\longrightarrow}\int_\Omega z_\infty,\qquad\psi'(t)\ge 0\quad\text{ for a.e. }t>0.
\end{equation*}

This implies that
\begin{equation*}
\liminf_{t\rightarrow\infty}	\partial_tz(t)=0\quad\text{in } L^1(\Omega).
\end{equation*}

Choosing $v=d$ as test function in \eqref{ivob} we obtain, for a.e. $t\in(0,\infty)$,
\begin{equation*}
	\int_\Omega\partial_t z(t)(d-z(t))+\int_\Omega\bs b\cdot\nabla z(t)(d-z(t))\ge
	\int_\Omega f(t)(d-z(t))
\end{equation*}
and so
\begin{equation*}
	\int_\Omega\partial_t z(t)(d-z(t))+|\bs b|\int_\Omega (d-z(t))\ge f(t)
	\int_\Omega (d-z(t)).
\end{equation*}

Since $d\ge z(t)$, taking $\displaystyle\liminf_{t\rightarrow\infty}$ to both sides of the inequality and using the assumption \eqref{liminf}, we obtain
\begin{equation*}
|\bs b|\int_\Omega(d-z_\infty)\ge(|\bs b|+2D)\int_\Omega(d-z_\infty),
\end{equation*}
which is a contradiction unless $z_\infty = d$.

Consider the following subsets of $Q_\infty=\Omega\times(0,\infty)$
\begin{equation*}
\Lambda=\big\{-d<z<d\big\},\quad I^+=\big\{z=d\big\},\quad
	I^-=\big\{z=-d\big\}.
\end{equation*}

Since $z$ solves the two obstacle problem \eqref{ivob}, it is well known that the following inequalities are verified a.e. in $Q_\infty$:
\begin{align*}
	&\partial_t z+\bs b\cdot\nabla z=f\quad\text{in}\quad\Lambda,\quad\partial_t z+\bs b\cdot\nabla z\leq f\quad\text{in}\quad I^+,\quad\partial_t u+\bs b\cdot\nabla u\geq f\quad\text{in}\quad I^-.
\end{align*}

If there is no finite time stabilization of the solution, since $z(t)$ is increasing in time, we may find a point $(x_0,t_0)$ and an  open subset $\omega_0$ of $\Omega$ with $x_0\in\omega_0$, such that,
$(x,t)\in\Lambda\cup I^-$ for  $t>t_0$. So,
\begin{equation*}
	f(t)\leq \partial_t z(x,t)+\bs b\cdot\nabla z(x,t)\quad\text{ for a.e.  }(x,t)\in\omega_0\times[t_0,+\infty).
\end{equation*}
Then, for any $t\ge t_0$ and any open set $\omega\subset\omega_0$, we have
\begin{align*}
\int_t^{t+1}	f(\tau)&\le\frac1{|\omega|}\int_t^{t+1}\int_\omega\big(\partial_t z(x,\tau)+\bs b\cdot\nabla z(x,\tau)\big)\\
&\le\frac1{|\omega|}\int_\omega \big( z(x,t+1)-z(x,t)\big)+|\bs b|\le
\big(2D+|\bs b|\big).
\end{align*}
As a consequence,
\begin{equation*}
\liminf_{t\rightarrow\infty} f(t)\le\liminf_{t\rightarrow\infty}\int_t^{t+1}f(\tau)d\tau\le 2D+|\bs b|
\end{equation*}
and this is a contradiction with  \eqref{liminf}. So $z(t)$ must stabilize in finite time.
\end{proof}

\section*{Acknowledgments}
This research was  partially supported  by CMAT - ``Centro de Matem\'atica da Universidade do Minho'', financed by FEDER Funds through ``Programa Operacional Factores de Competitividade - COMPETE'' and by
Portuguese Funds through FCT, ``Funda\c c\~ao para a Ci\^encia e a Tecnologia'',
within the Project PEst-OE/MAT/UI0013/2014.

\vspace{1cm}

\begin{small}

\begin{tabular}{ll}
\noindent Jos\'e Francisco Rodrigues&\hspace{2,5cm}Lisa Santos\\
jfrodrigues@ciencias.ulisboa.pt&\hspace{2,5cm}lisa@math.uminho.pt\\
CMAF-FC\_ULisboa,&\hspace{2,5cm}CMAT / Dept. of Mathematics and Applications\\
Av. Prof. Gama Pinto, 2,&\hspace{2,5cm}Campus de Gualtar,\\
1649-003 Lisboa, Portugal&\hspace{2,5cm}4710-057 Braga, Portugal\\
\end{tabular}

\end{small}

\end{document}